\documentclass{article}
\usepackage[T1]{fontenc}
\usepackage[utf8]{inputenc}
\usepackage{amsfonts, amsmath, amssymb, amsthm}
\usepackage{graphicx}
\graphicspath{{./fig/}}
\usepackage{caption}
\usepackage{subcaption}
\usepackage{wrapfig}
\usepackage{placeins}
\usepackage[hidelinks]{hyperref}
\usepackage{float}
\usepackage[title]{appendix}
\usepackage[a4paper, total={5in, 10in}]{geometry}
\usepackage{xcolor}
\usepackage[raggedrightboxes]{ragged2e}

\numberwithin{equation}{section}

\newtheorem{theorem}{Theorem}[section]
\newtheorem{corollary}[theorem]{Corollary}
\newtheorem{definition}[theorem]{Definition}
\newtheorem{lemma}[theorem]{Lemma}
\newtheorem{proposition}[theorem]{Proposition}
\newtheorem{remark}[theorem]{Remark}
\usepackage[sorting=none, maxbibnames=99]{biblatex}
\newcommand{\eps}{\varepsilon}
\newcommand{\poincare}{Poincar\'e }
\newcommand{\PB}{Poincar\'e-Bendixson }
\newcommand{\PH}{Poincar\'e-Hopf }
\DeclareMathOperator{\Div}{div}
\DeclareMathOperator{\Int}{int}
\DeclareMathOperator{\diam}{diam}

\title{\vspace{-0.7in}Thick Arnold tongues\footnote{ML acknowledges partial support  by the NSF grant DMS-9704554
}}

\author{Mark Levi and Alexey Okunev}

\begin{document}
\vspace{-5cm}
\maketitle

\begin{abstract}
 We introduce and study a physically motivated problem that exhibits interesting and perhaps unexpected mathematical features. A cellular flow is a two-dimensional Hamiltonian flow of the Hamiltonian $H(x, y) = \cos(x) \cos(y)$.
 We study a simple model of the dynamics of an inertial particle carried by such a flow, subject to viscous drag and to an additional constant external force $(b, a)$.
 In the limiting case of zero inertia particles  the dynamics is Hamiltonian with $H(x, y) = \cos(x) \cos(y) - ax + by$.
 For small but nonzero $a, \ b $ there appear ``channels" of trajectories that wind their way to infinity, of small relative measure,  while most trajectories remain periodic.

 By contrast, for nonzero inertia, no matter how small, almost all particle trajectories drift to infinity.
 Moreover, the asymptotic direction of this drift no longer coincides with the direction of forcing, and rather becomes   Cantor-like function of the forcing direction $a/b$, and with an unexpected feature: the plateaus of this function occupy a set of full measure. Moreover, the complement to this set has zero Hausdorff dimension. In a two-parameter representation (one parameter being the forcing direction $a/b$, the other  the drag coefficient), this gives rise to Arnold tongues,  the tongues corresponding to rational slopes of drift. However, unlike Arnold's example, the complement to the union of all tongues has zero measure. This is explained by the behavior of rotation number for monotone families of circle maps with flat spots.
\end{abstract}

\section{Introduction and Results}
\subsection{Background} \label{s:intro-results}

The problem of particle transport by fluid flow has been studied extensively - in experiment, theory and computation. The case of spherical particles in Newtonian fluids has been modeled by the so-called Maxey--Riley equation (\cite{maxey1983equation}), which is Newton's second law for the particle accounting for buoyancy, drag, inertial effect of the fluid (added mass), and boundary layer (memory)  (\cite{haller3, haller2} and references therein).
Many authors consider a simplified form of this equation, where the only force coming from the fluid is the drag proportional to the velocity mismatch between the particle and surrounding fluid:
\begin{equation} \label{e:MR}
	\ddot {\bf x} = - \frac{1}{\eps} ( \dot {\bf x} - {\bf v} ( {\bf x} ) ),
\qquad ({\bf x}, \dot {\bf x}) \in \mathbb R^4.
\end{equation}
Here, ${\bf x}(t) \in \mathbb R^2$ is the particle position, ${\bf v}(\bf x)$ is the fluid velocity at ${\bf x}$, and $\eps = m / k$, where $m$ is the particle mass and $k$ is the drag coefficient (see Appendix~\ref{a:gravity} for details).
The smaller $\eps$, the smaller  the impact particle's inertia has on the dynamics.
When $\eps=0$ (zero-inertia case) the particle just moves together with the surrounding fluid: $ \dot {\bf x} = {\bf v} ( {\bf x} ) $; we will discuss what happens when $\eps$ is small but positive.

Many studies on particle transport deal with the particular case when fluid motion is given by the cellular flow (also called the 2-D Taylor-Green vortex flow), i.e., the Hamiltonian flow along the contour lines of $H = \cos x \cos y$, Figure~\ref{fig:cells}.
Cellular flow is often used to model particles (e.g., plankton or microplastic) carried by convective motions near sea surface (\cite{arrieta2015microscale, baggaley2016stability, kreczak2023dynamics}) or as a very simplified model of turbulence (\cite{maxey1986gravitational}).
The dynamics of~\eqref{e:MR} with ${\bf v}$ being a cellular flow is quite simple, as Figure~\ref{fig:cellsparticle} illustrates: the particles approach the cell boundaries. Additional effects can make the problem quite nontrivial, and
  there is a vast literature on the subject.
This includes studying the dynamics of  particles influenced by random thermal noise
(\cite{sarracino2016nonlinear, cecconi2017anomalous, renaud2020dispersion, nath2022transport}),
the roles of  particle shapes,
(asymmetric~\cite{shin1997chaotic}, compound~\cite{piva2003single},
self-propelled~\cite{torney2007transport, yin2024diffusion}),
particles driven by a magnetic force (\cite{ravnik2014numerical}),
particles in a time-dependent cellular flow
(\cite{fung1997gravitational, chan1999change, venditti2022invariant}).
The list above is by no means exhaustive; additional references can be found in the cited papers.

The case of particles in the cellular flow subject to a constant external force (gravity) was considered as early as 1949 by Stommel~\cite{stommel1949trajectories}, who considered the zero-inertia case.
The positive inertia case was considered by Maxey and Corrsin in~\cite{maxey1986gravitational} for aerosol particles (i.e., particles with the density much greater than that of the fluid).
They carried out a numerical investigation of this problem for a wide range of parameters and noticed two surprising effects.
First, cellular flow makes particles settle much faster than in the still fluid.
Second, particles equidistributed initially tend to concentrate on certain curves.
The case of aerosol particles was then rigorously investigated in~\cite{rubin1995settling} for a special case when the gravity direction is aligned with the flow cells orientation, and particle concentration was proved.
In~\cite{maxey1987motion}, the case of arbitrary particle density was considered.
Curiously, when the particle and fluid density are of the same order, chaotic dynamics was observed even without the external forcing~\cite{wang1992chaotic}.
These results are nicely summarized in the review~\cite[\S3]{maxey2017modeling}.

\begin{figure}[t]
\centering
\begin{subfigure}{.45\textwidth}
 \centering
 \includegraphics[height=35mm]{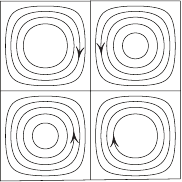}
 \caption{A cellular area-preserving flow in ${\mathbb R} ^2 $ given by the Hamiltonian $H(x,y)= \cos x \cos y$}
 \label{fig:cells}
\end{subfigure}
\hspace{10mm}
\begin{subfigure}{.4\textwidth}
 \centering
 \includegraphics[height=35mm]{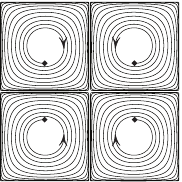}
 \caption{Trajectories of particles carried by (\ref{e:MR}) where $ {\bf v} $ is an in Figure \ref{fig:cells} and $\eps=1/25$.}
 \label{fig:cellsparticle}
\end{subfigure}
\caption{A particle in a cellular flow}
\label{fig:cells-all}
\end{figure}

Mathematically, accounting for a constant external force acting on an aerosol particle (as in~\cite{maxey1986gravitational, rubin1995settling}) can be reduced   to adding a constant term to the fluid velocity, i.e. to taking ${\bf v}({\bf x}) = {\bf u}({\bf x}) + {\bf w}$ in~\eqref{e:MR}, where ${\bf u}({\bf x})$ is the fluid velocity, and ${\bf w}$ is the terminal velocity of a particle pulled by gravity in a still liquid.
This reduction is explained in~\cite{maxey1986gravitational}, it is also summarized in Appendix~\ref{a:gravity}.
 One can write the resulting ${\bf v}({\bf x})$ as a Hamiltonian vector field by adding a linear part to the Hamiltonian giving the fluid motion ($\bf u$ must be Hamiltonian if the fluid is incompressible). If ${\bf w} = (b, a)$, we get
\begin{equation} \label{e:fluid}
{\bf v} = \bigg(\frac{\partial H}{\partial y}, -\frac{\partial H}{\partial x} \bigg), \ \ \hbox{with} \ \
H(x, y) = \cos x \cos y - ax + by.
\end{equation}
The dynamics of the resulting vector field $\dot {\bf x} = {\bf v}({\bf x})$ is referred to as zero-inertia case (\cite{maxey1986gravitational}), it results from the assumption that the particle's relative velocity in the fluid equals the terminal velocity $\textbf{w}$ so that gravity is compensated by the fluid drag.

\begin{wrapfigure}{r}{0.5\textwidth}
  \center{ \includegraphics[width=0.9\linewidth]{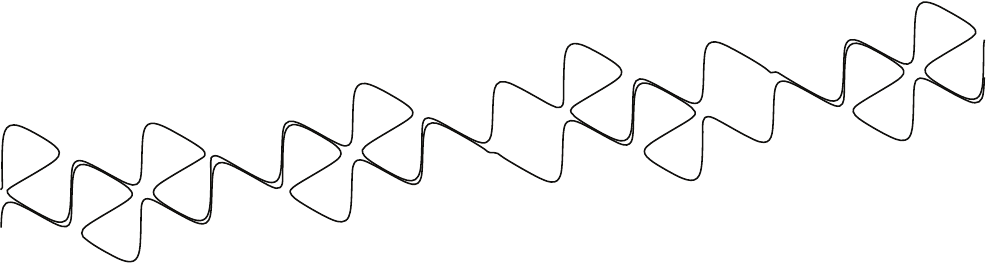}}
  \caption{Two trajectories of (\ref{e:MR}) where ${\bf v} $ has a 3-fold symmetry. The trajectories are separated by stable manifolds of equilibrium points. }
  \label{fig:triangulargrid2}
\end{wrapfigure}

\subsection{Subject of this paper}
We give a rigorous analysis of the model~\eqref{e:MR} with ${\bf v}$ given by~\eqref{e:fluid} focusing on the direction of particle drift.
Note that one can also consider Hamiltonians with the periodic part other than $\cos x \cos y$ in~\eqref{e:fluid} using similar methods.
As an example, Figure~\ref{fig:triangulargrid2} shows trajectories of~\eqref{e:MR} where $\cos x \cos y$ in~\eqref{e:fluid} has been replaced by a Hamiltonian with a 3-fold symmetry.\footnotemark
Let us consider first the ``unperturbed'' case  of $\eps=0$, corresponding to the Hamiltonian vector field (\ref{e:fluid}).  Trajectories are the level curves of $H$. For  the periodic case of $ a=b=0 $ shown in Figure~\ref{fig:cells} the picture is trivial; but case of nonzero $a, b$, illustrated in Figure~\ref{f:level-lines} in Section~\ref{s:Ham}, is already quite interesting combinatorially.

Changing  $(a,b)$ from zero gives rise to ``channels''  consisting of unbounded trajectories winding their way to infinity between homoclinic cells (Figure~\ref{f:level-lines} in Section~\ref{s:Ham}). These trajectories have  interesting  combinatorics related to continued fraction expansion of the slope $ b/a $.  This problem has been studied in  \cite{novikov1982hamiltonian, arnold1991topological, zorich1999wind, dynnikov2022chaotic} in a somewhat more general setting.

\footnotetext{We took
  $H(x, y) = h({\bf x})h(R{\bf x})h(R^2 {\bf x}) - ax + by$ with
  ${\bf x}=(x,y)$, $h({\bf x}) = \sin x \sin y$ and with $R$ being the rotation around the origin by $2 \pi /3$.
}

The dynamics changes drastically for $\eps>0 $.
Figure~\ref{fig:drift1} shows typical   trajectories of (\ref{e:MR})  for small $\eps$; in all numerical simulations  these trajectories drift in a rational direction regardless of the direction of imposed forcing. Moreover, the drift slope $m$  can be sensitive to the slope $ \alpha = a/b $  of imposed forcing (compare the two top cases in Figure~\ref{fig:drift1}: half a degree change in forcing direction causes a drastic change of the trajectory).
This figure, as well as other numerical experiments, may give a false impression of discontinuous dependence of the drift slope $m$   on  the slope $ \alpha $  of imposed forcing.

Actually, we will later see that $ m ( \alpha ) $  {\it is} continuous; the  false impression of discontinuity is explained by the fact that $m$ is  a Cantor-like function that is locally constant on a full measure set.  Thus the slope of trajectories changes on the set of slopes $\alpha=a/b$ of zero measure (Figure~\ref{f:staircase} below), giving the impression of discontinuity.
\FloatBarrier
\begin{figure}[H]
  \centering
  \includegraphics[width=0.8\textwidth]{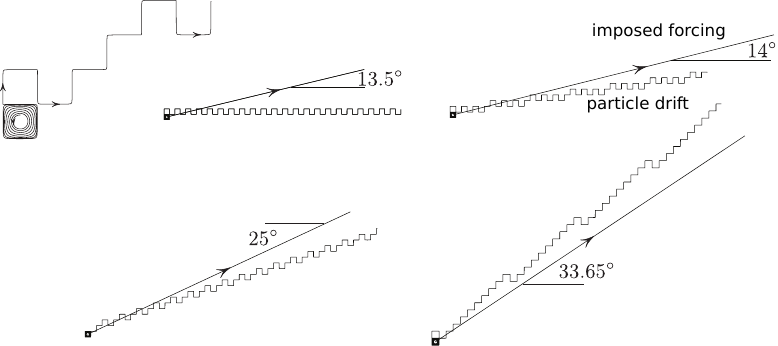}
  \caption{Trajectories of the system~\eqref{e:MR},~\eqref{e:fluid} with $\sqrt { a ^2 + b ^2 } = 0.02 $ and $ \eps = 1/25 $.}
  \label{fig:drift1}
\end{figure}

\subsection{Main results}
Consider the system~\eqref{e:MR} with the vector field ${\bf v}$ given by~\eqref{e:fluid}. This is a model of a particle carried by the cellular flow and pulled by a constant forcing $(b, a)$ such as gravity proposed by Maxey and Corrsin~\cite{maxey1986gravitational}.
Our main result describes   the dependence of the asymptotic slope $m$ of particles' forward trajectories  on the forcing slope $\alpha = \frac a b$.

For definiteness, we assume that $b$ and $\eps$ are fixed, so that $a$ is tied to  $\alpha$ via $a = \alpha b$. Also,    $a$ and $b$ will be non-zero and  sufficiently small but fixed, but $\varepsilon$ is much smaller, as is stated in the theorem below. By  symmetry, it suffices to consider the case $b \ge a > 0$. We relax the condition $b \ge a$ to $a \le 1.5b$ to allow the forcing slope $\frac a b$ to pass through $1$.

\begin{theorem} \label{t:main}
There exists $\gamma>0$ such that for any positive $\delta < \gamma$ there exists $\varepsilon_0 > 0$ such that, provided that $a$ and $ b  \in [\delta, \gamma]$ with $a \le 1.5b$ and $\eps \in (0, \eps_0)$, the following holds for~\eqref{e:MR} with ${\bf v}$ given by~\eqref{e:fluid}.
\begin{enumerate}

 \item \label{i:rho-exists} For all initial data $ ({\bf x}_0, \dot {\bf x}_0 )\in {\mathbb R}^4$ with the exception of a countable union of codimension one hypersurfaces the particle's forward trajectory $ {\bf x} (t) $ in $ {\mathbb R} ^2 $ is unbounded and follows some ray at a finite distance. Moreover, the slope $m$ of this ray is the same for all such initial data. We will call this $m$ the   drift slope.

 \item The   drift slope $ m $ has a Cantor-like dependence on the direction $ \alpha = \frac a b $ of the forcing (Figure~\ref{f:staircase}). To be precise, we fix $b \in [\delta, \gamma]$, $\eps \in (0, \eps_0)$ and set $a(\alpha) = \alpha b$,
 $I = [\delta/b, \; \min(\gamma/b, \; 1.5)]$. Then
 $m(\alpha), \; \alpha \in I$
 is a Cantor-like function: it is continuous, non-decreasing, and the set $C \subset I$ of points where $m$ is not locally constant is a nowhere dense perfect set.
 Also,
 $C = \{\alpha \in I : m(\alpha) \not \in \mathbb Q\}$,
 and each rational value $m(\alpha) = p/q$ achieved on $I$ is achieved on a closed nondegenerate interval (plateau) $m^{-1}(p/q)$.

 \item
 The set $C$ has zero Hausdorff dimension (and hence also zero Lebesgue measure). According to this and the previous item,  $m(\alpha)$ is rational for almost all values of the parameter $\alpha$.

 \item The function $m(\alpha)$ is "steep to all orders" at the endpoints of each plateau $m^{-1}(p/q)$,  Figure~\ref{fig:steepness}.
 More precisely, let $\alpha_0$ be  the right endpoint of a plateau and let $\Delta m = m(\alpha_0+\Delta \alpha) - m(\alpha_0)$.
 Then, there exist $c > 0$ and $\delta > 0$ such that
 $\Delta \alpha < e^{-\frac{c}{\Delta m}}$ when $\Delta \alpha \in (0, \delta)$.
 A similar statement holds for the left endpoint of every plateau.

 Note that this inequality implies that for any $n \in \mathbb N$ we have
 $(\Delta m)^n \gg \Delta \alpha$ when $\Delta \alpha \to 0^+$ -- that is, the graph of $m(\alpha)$ is tangent to all orders to vertical lines passing through the plateau's endpoints.
\end{enumerate}
\end{theorem}
\begin{figure}[t]
\centering
\begin{subfigure}{.4\textwidth}
 \centering
 \includegraphics[width=\linewidth]{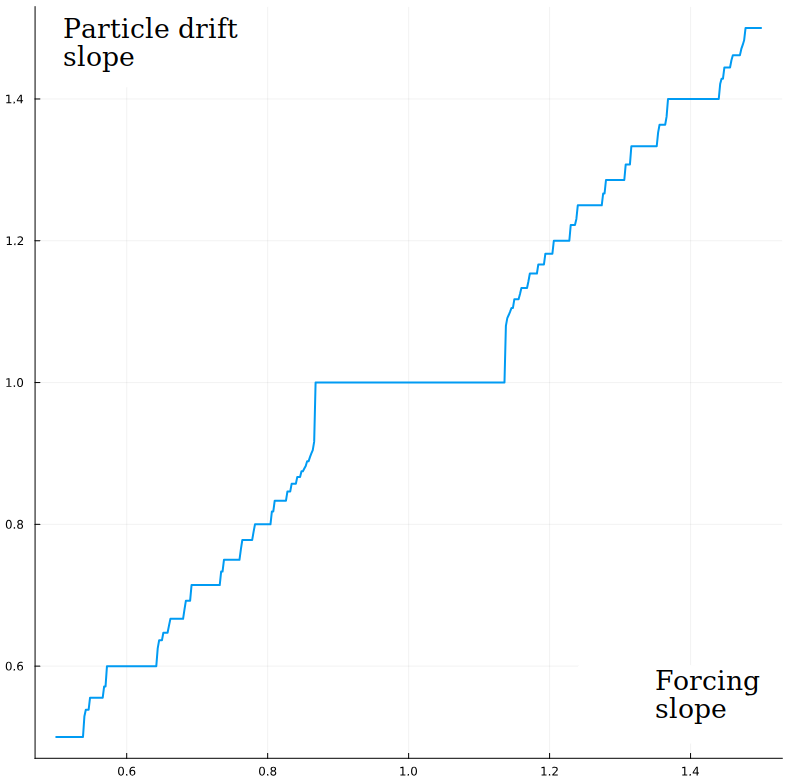}
 \caption{Cantor-like dependence of the particle drift slope on the direction of external force. Particle's drift direction is rational for Lebesgue-almost all directions of external force.}
 \label{f:staircase}
\end{subfigure}
\hspace{20mm}
\begin{subfigure}{.4\textwidth}
 \centering
 \includegraphics[width=\linewidth]{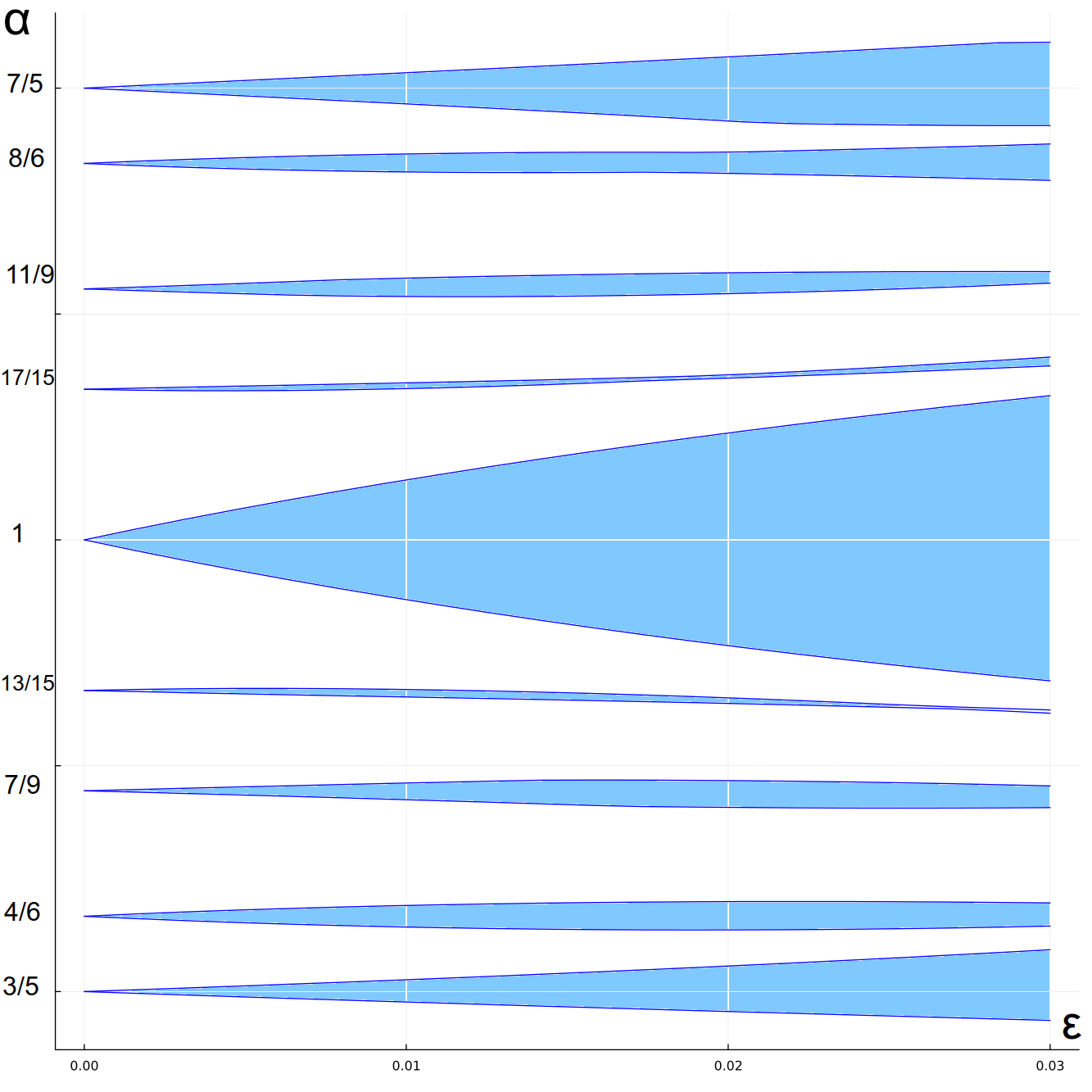}
 \caption{Arnold tongues~\cite{arnold1961small}, i.e., parameter values corresponding to a fixed rational particle drift slope, shown in blue. If all the tongues were shaded, they would occupy full measure.}
 \label{f:tongues}
\end{subfigure}
\caption{Particle drift direction}
\end{figure}

\begin{wrapfigure}{r}{0.6\textwidth}
 \center{ \includegraphics[width=0.9\linewidth]{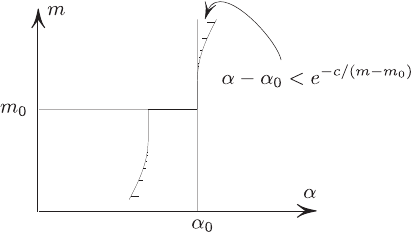}}
 \caption{Steepness of $ m(\alpha)$ near the ends of plateaus. }
 \label{fig:steepness}
\end{wrapfigure}
\noindent Particle drift slope $m(\alpha)$ can be described in terms of the rotation number of a family of circle maps, see Section~\ref{s:flat-spots} below.
It is a classical fact that rotation number of a typical monotone one-parameter family of circle diffeormorphisms is  a Cantor-like function with a plateau at each rational value;
a model example is the Arnold family~\cite{arnold1961small}  $x \mapsto x + \alpha + \eps \sin x$.
However, $m(\alpha)$ being rational on a full measure set of parameters is in marked contrast with the Arnold family,  where the rotation number is believed to be irrational on a set of positive measure for small $\eps$,~\cite{ecke1989scaling}.
This property  was observed for Cherry flows on the 2-torus~\cite{boyd1985structure}.
An interesting feature of our setting is that  while for $\eps = 0$ the measure of the set of $\alpha$ with   $m(\alpha) $  irrational is full, this measure collapses to zero as soon as $\varepsilon>0$, no matter how small.  That is, the blue tongues in Figure~\ref{f:tongues} occupy full measure (as Theorem~\ref{t:main} states), justifying the term "thick Arnold tongues".

It should be noted that even the unperturbed ($\varepsilon=0$) dynamics, that of the basic Hamiltonian system with the periodic+linear Hamiltonian~\eqref{e:fluid}, is suprisingly nontrivial.
As a side result of this article we describe the combinatorics of unbounded trajectories of this system, relating it to number-theoretic properties of $a/b$ in Section~\ref{s:Ham} below.
We should note that this reduced problem is a special case of a well-known problem studied by Arnold, Dynnikov, Novikov, Zorich and others~(\cite{novikov1982hamiltonian, arnold1991topological, zorich1999wind, dynnikov2022chaotic} and references therein): to describe the foliation of a closed surface given by a closed 1-form. In our case, the surface is the torus $\mathbb R^2 / 2\pi \mathbb Z^2$, and the 1-form is $dH$.

\subsection{Monotone families of circle maps with flat spots} \label{s:flat-spots}
In this section  we briefly explain the underlying dynamical mechanism behind the behavior of the particle drift slope described in Theorem~\ref{t:main}.

The particle dynamics given by~\eqref{e:MR} and~\eqref{e:fluid} can be reduced to a continuous circle map that is constant on two intervals and is monotone increasing elsewhere.
This reduction is done in Section~\ref{s:particle}.
Briefly, the flow of~\eqref{e:MR} in $\mathbb R^4$ has a globally attracting two-dimensional normally hyperbolic invariant manifold. This manifold  inherits the periodicity of  $ {\bf v}  $ in $ {\bf x} $ so that the problem reduces to the flow on a $ 2 $-torus.  This flow admits a transversal circle, and the resulting \poincare map of the circle   captures the entire dynamics of (\ref{e:MR}) in $ {\mathbb R} ^4 $.

This first return map turns out to be monotone increasing and continuous apart from  two jump discontinuities, as explained later in Section~\ref{s:poinc} (Figure~\ref{f:homoclinic-open}).
In particular, the particle drift slope is the rotation number of the circle map (up to a linear transformation).
So the dependence of the drift slope on the forcing slope is captured by the dependence of the rotation number on the parameter in a one-parameter family of circle maps.
For convenience we will consider the family of inverse maps extended to the jump intervals as a constant, i.e. a family of continuous degree one circle maps with flat spots.
The properties of this family are captured by the definition below.
Very similar\footnote{Up to some differences such as that expansivity condition can be removed, certain conditions on the derivative near the flat spot (such as having bounded variation) can be required, and often the case of just one flat spot is considered.} families naturally occur in the study of Cherry flows on torus and of truncations of non-invertible circle endomorphisms.
They were studied by Boyd, Boyland, Graczyk, Palmisano, {\'S}wi\k{a}tek, Veerman, and others (\cite{boyd1985structure, swikatek1989endpoints, veerman1989irrational, graczyk2017scalings, palmisano2021phase, boyland2024abundance} and references therein).

\begin{definition}[Monotone family of circle maps with flat spots] \label{d:monotone}
  Let $f_s$ be a one-parameter family of circle maps $f_s: S^1\mapsto S^1$, where $S^1 = \mathbb R / \mathbb Z$ defined for $s \in [s_{min}, s_{max}]$ and continuous with respect to both $x$ and $s$.
  Let $\tilde f_s: \mathbb R \mapsto \mathbb R$ be a lift of $f_s$.
  We will say that $f_s$ is a \emph{monotone family}, if
  \begin{enumerate}
   \item $\tilde f_s(x)$ is continuous and non-decreasing with respect to both $x$ and $s$.
   \item $f_s$ has degree $1$ for all $s$.
   \item There are $m$ flat spots for each $s$, where $ m\geq 1 $ is an integer independent of $s$. More precisely, there exists $m \ge 1$ and $m$ disjoint closed arcs $I_j(s)$ such that $f_s|_{I_j}$ is constant. The endpoints of these arcs depend on $s$ continously.
   \item \label{prop:expand} There exists $\lambda > 1$ such that for each $s$ the map $f_s$ is $\lambda$-expanding outside $\cup_{j=1}^{m} I_j$: $\tilde f_s(x)$ is differentiable in $x$ and
   $\frac{d}{dx} \tilde f_s(x) > \lambda$.
   \item \label{prop:flat-spots} Pick any lifts $\tilde I_j$, $j=1, \dots, m$ of the flat spots and let $\tilde b_j(s) = \tilde f_s(\tilde I_j(s))$ be the ``height" of the plateau above $\tilde I_j$. Then, each of these heights is differentiable with respect to $s$ and increases with the speed separated from zero: there exists $\nu > 0$ such that $\frac{d}{ds} \tilde b_j(s) > \nu$ for all $s$, $j=1, \dots, m$.
  \end{enumerate}
\end{definition}

The same argument as for circle homeomorphisms shows that such maps have a rotation number $\rho(s)$ which is also continuous and monotone in the parameter $s$, and the rotation number is rational if and only if there is a periodic orbit. Moreover, one can show that the rotation number is rational if and only if the image of the flat spot intersects itself, so a rational rotation number corresponds to a whole interval of the values of the parameter $s$. Hence, the graph of $\rho(s)$ is a Cantor-like function increasing when $\rho$ is irrational and with plateaus at rational values of $\rho$. This is also common for circle homeomorphisms, e.g. for the Arnold family $f_s(x) = x + s + \eps \sin x$, $\eps \in (0, 1)$. A remarkable constrast with the case of homeomorphisms is that for maps with flat spots the set of the values of $s$ with irrational rotation number has zero Lebesgue measure and, moreover, zero Hausdorff dimension. This was first proved by Boyd for a special case of families of the form $f_s(x) = f_x + s$ and then for generic classes of families of maps with \emph{one flat spots} by {\'S}wi{\k{a}}tek~\cite{swikatek1989endpoints} and Veerman~\cite{veerman1989irrational}.
Veerman also remarked that this statement can be generalized for any number of flat spots.
We do this by extensive use of the ideas from~\cite{veerman1989irrational}.\footnote{
   Instead of the expansivity (Property~\ref{prop:expand} in Definition~\ref{d:monotone}) Veerman~\cite{veerman1989irrational} requires that $\ln \big( \frac{d}{dx} f_s \big)$ has bounded variation outside the flat spot and uses this condition to prove a weaker form of expansivity.
   In the family originating from particle's dynamics this variation is unbounded by Remark~\ref{r:sing-derivative} below, but expansivity holds, which motivates this difference compared to the setting in~\cite{veerman1989irrational}. Also, this family has two flat spots, so we need to consider the case of more than one flat spot.
 }.
\begin{theorem} \label{t:rotation-numbers}
 Suppose $f_s$, $s \in [s_{min}, s_{max}]$ is a monotone family of circle maps with flat spots, and $\tilde f_s$ is its lift to $\mathbb R$.
 Then
\begin{enumerate}
 \item The rotation number
  \[
    \rho(s) = \lim_{n \to \infty} \frac{(\tilde f_s)^n(\tilde x) - \tilde x}{n}
  \]
  exists\footnote{That is, the limit exists and does not depend on the initial point $\tilde x \in \mathbb R$.}. Moreover, $\rho(s)$ is continuous and non-decreasing.
 \item For any $\frac p q \in \mathbb Q$, its preimage $\rho^{-1}(p/q)$ is a closed interval with positive length (unless it is an empty set).
 For any $r \not \in \mathbb Q$, its preimage $\rho^{-1}(r)$ is a point (unless it is an empty set).
 \item The set $C = \rho^{-1}([s_{min}, s_{max}] \setminus \mathbb Q)$ is a nowhere dense perfect set.
 \item The Hausdorff dimension  $\dim_H(C) = 0$.
 \item
 The function $\rho(s)$ is ``steep to all orders'' at the endpoints of each plateau $\rho^{-1}(p/q)$, as described in Theorem~\ref{t:main}.

\end{enumerate}
\end{theorem}

\noindent This result explains the properties of $\rho(\alpha)$ claimed in Theorem~\ref{t:main}. It is proved in Section~\ref{s:dim-H} relying extensively on the ideas from~\cite{veerman1989irrational}.

\section{The Hamiltonian flow} \label{s:Ham}
Before studying particle transport in the sections that follow, let us make some observations on trajectories of the Hamiltonian flow~\eqref{e:fluid}:
\begin{equation} \label{e:fluid-ode}
 \dot x = -\cos x \sin y + b, \qquad
 \dot y = \sin x \cos y + a.
\end{equation}
For $a = b = 0$ the phase portrait is very simple, consisting of the square  grid of separatrices (Figure~\ref{fig:cells}) with interiors of the squares filled by periodic orbits. Saddles form the lattice $(\frac \pi 2 + \pi k_1, \frac \pi 2 + \pi k_2)$, where $(k_1, k_2) \in \mathbb Z^2$.
\begin{figure}[H]
 \centering
 \includegraphics[scale=0.3]{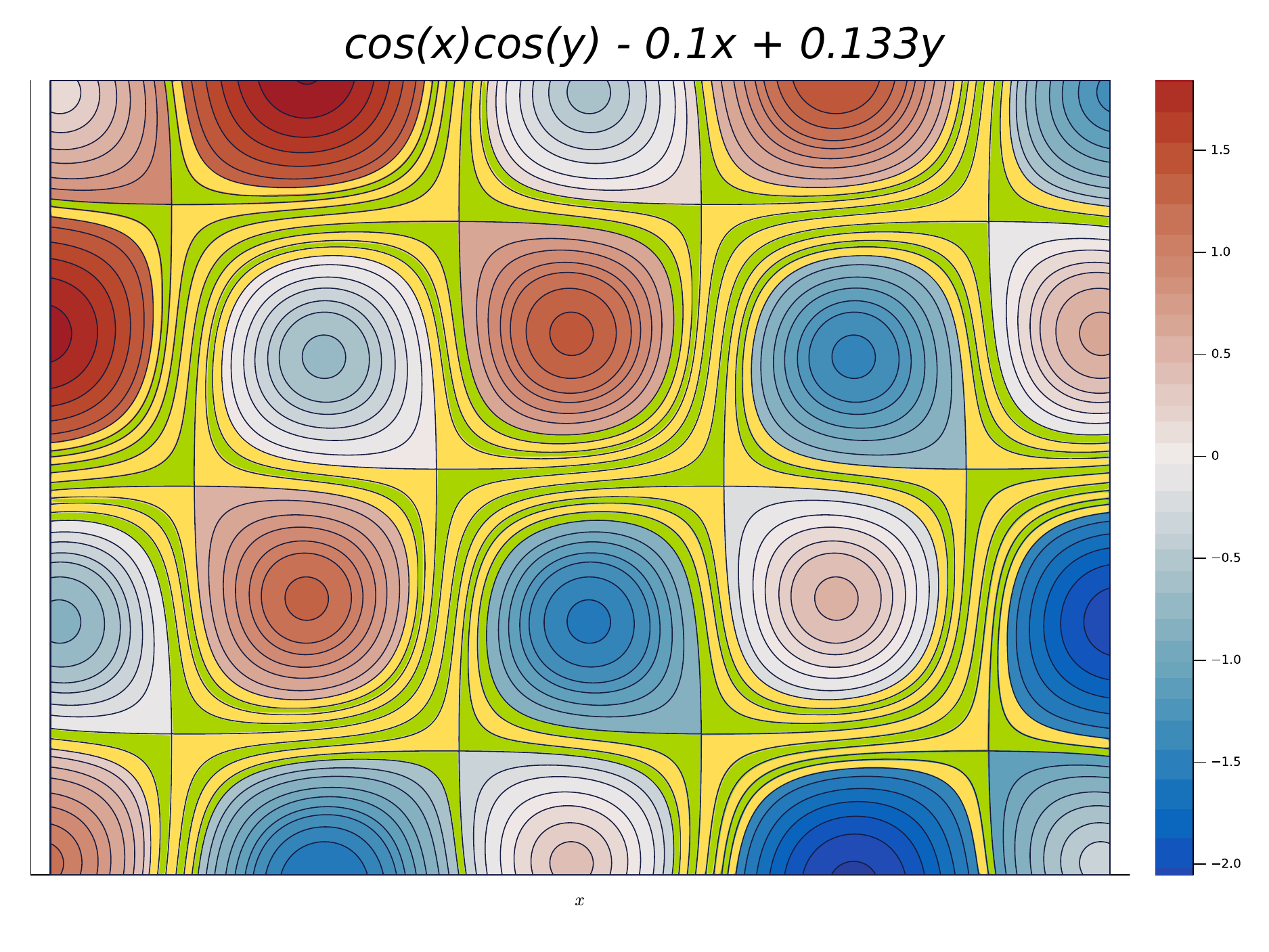}
 \caption{Level lines of the Hamiltonian~\eqref{e:fluid}. ``Channels" are highlighted in yellow and green.}
 \label{f:level-lines}
\end{figure}
\noindent For $a, b > 0$ the phase portrait becomes somewhat nontrivial despite the simplicity of the Hamiltonian.
For small $ a $, $b$ most of the phase space will still be filled by periodic trajectories bounded by homoclinic loops; however, certain ``channels"  appear, (Figure~\ref{f:level-lines}); these channels hug the original grid of the flow with $a=b=0$, and all forward trajectories in these channels are now unbounded, with the exception of a zero measure set of those lying on stable manifolds of saddles.
Any such unbounded trajectory stays within a bounded distance from the straight line
$by-ax=0$ (Figure~\ref{f:chess}), since $by-ax = H - \cos x \cos y$, where $H$ is constant along a trajectory, and $\cos x \cos y$ is bounded.
So, the asymptotic direction of Hamiltonian trajectories is the same as the direction of the forcing vector $(b, a)$ in~\eqref{e:fluid-ode}. We shall see later that this is decidedly not the case for particles carried by the flow.

\begin{figure}[h]
\centering
\begin{subfigure}{.5\textwidth}
 \centering
 \includegraphics[width=0.9\linewidth]{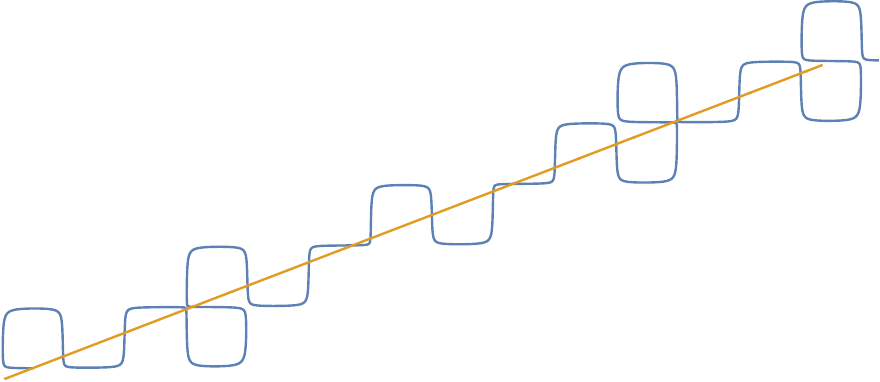}
 \caption{A trajectory follows a line with the slope $a/b$}
 \label{f:chess}
\end{subfigure}
\hspace{5mm}
\begin{subfigure}{.3\textwidth}
 \centering
 \includegraphics[width=\linewidth]{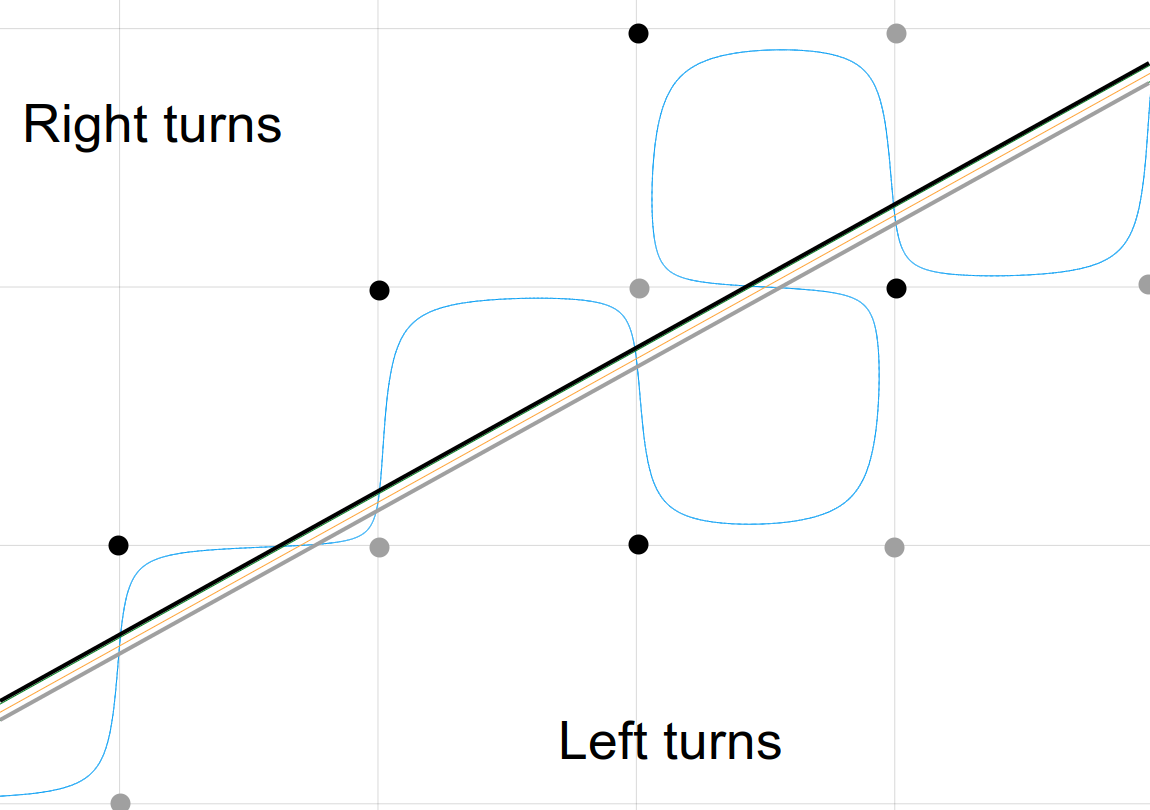}
 \caption{Turn rules}
 \label{f:2-lines}
\end{subfigure}
\caption{ Fluid trajectories given by~\eqref{e:fluid-ode} }
\end{figure}

\medskip
\subsection*{Combinatorics of Hamiltonian trajectories}
{\bf A ``chess'' game.} We now describe the combinatorics of unbounded Hamiltonian trajectories. This part is not used in the analysis of particle dynamics that leads to   Theorem~\ref{t:main}.

For small $a$ and $b$, any unbounded trajectory is close to the net of separatrices of the unperturbed system with $a=b=0$, (Figure~\ref{f:chess}), and can be approximated by a broken path formed by abutting edges of the net. Every trajectory approaching a vertex makes a turn, either to the left or to the right depending on which side of the stable manifold the trajectory enters the saddle's neighborhood (the trajectories we consider are unbounded and hence they do not lie on the stable manifolds of the saddles).
The following ``chess rule'' gives a simplified combinatorial description of unbounded trajectories (such as one in Figure~\ref{f:chess}). This description is valid for small $a$ and $b$; to be specific, we take $ |a|, \  |b|< 0.1 $.
\vskip 0.1 in
Consider a chess coloring of the vertices of the square lattice $\{\frac \pi 2 + \pi \mathbb Z\}^2$ of the saddles of the Hamiltonian with $ a=b=0$.  Namely, we call a vertex $(\frac \pi 2~+~\pi k_1, \frac \pi 2~+~\pi k_2)$ \emph{odd} if $k_1+k_2$ is odd and \emph{even} otherwise. Also, consider two oriented parallel straight lines $by-ax=c_o$ and $by-ax=c_e$\footnote{The constants $c_o$ and $c_e$ depend on the value of $H$ at the trajectory we are describing and are specified below in Appendix~\ref{a:saddles}.} corresponding to odd and even vertices, respectively (Figure~\ref{f:2-lines}).
Assume that both lines do not pass through grid vertices (this is true if the considered trajectory is unbounded).
Focusing our attention on the odd vertices first, the ``odd'' line  divides these vertices into the left class and the right class; we label these odd vertices by $L$ or $R$ according to their class. Similarly, even vertices are divided by the ``even'' line into left and right classes, and again we label these by $ L$ or $ R $ accordingly.   All vertices are now labeled, and we state {\it the path-generating rule}. Our path will consist of abutting horizontal and vertical edges connecting neighboring vertices. As we travel along an edge and meet a vertex, we turn right if the point is labeled $R$ and turn left if the point is labeled $L$.

The idea behind this description is very simple: to understand if a right or a left turn happens at a vertex, one needs to compare the {\it constant} value of $H$ along the trajectory with the value of $H$ at the saddle of~\eqref{e:fluid-ode} near this vertex. The precise proof is given in Appendix~\ref{a:saddles}.

\medskip

\subsection*{Area-preserving flows on surfaces and Novikov's problem.}
To put the Hamiltonian case in proper context, let us note that this is a particular case of the following much harder problem: study the foliation of a closed surface given by a closed 1-form.\footnote{The following equivalent (\cite{arnold1991topological}, see also~\cite{zorich1999wind}) problem is known as Novikov's problem on the semiclassical motion of an electron. Consider a level surface of some smooth function in the 3-torus $\mathbb T^3$. Lift it to a periodic surface in $\mathbb R^3$ and intersect it by an arbitrary plane $\mathbb R^2$. What can be said about the connected components of this intersection?} In our case, the surface is $\mathbb T^2$ and the $1$-form is $dH$.
For the torus case, Arnold~\cite{arnold1991topological} showed that,
given that the periods of the 1-form are not commensurate, under a mild genericity condition there exists a transversal such that the corresponding first return map is conjugate to an irrational rotation by the angle equal to the ratio of the periods. As in our case, there might also be topological discs formed by closed finite leaves of the foliation (homoclinic loops). For higher genus surfaces, the foliation can be much more complicated; this is known as \emph{chaotic regime}. This phenomenon was extensively studied by Dynnikov with several co-authors (cf.~\cite{dynnikov1997semiclassical, dynnikov2022chaotic} and references therein).
For more background on this problem, we direct the reader to the survey by Zorich~\cite{zorich1999wind}.
Another rich research area is mixing properties of area-preserving surface flows resticted to quasi-minimal components (such as the complement to the union of the interiors of homoclinic loops in our case, provided that $a/b$ is irrational).
An up-to-date discussion on this topic can be found in~\cite{fayad2023non}.

\section{From particle dynamics to a circle map} \label{s:particle}

\subsection{Reduction to a torus flow} \label{ss:fenichel}
\textbf{Two-dimensional attracting manifold.}
The simplified Maxey-Riley equation~\eqref{e:MR} is a singular perturbation of the flow $\dot{\bf x} = {\bf v}(\bf{x})$ in $\mathbb R^2$.
A general fact from singular perturbation theory commonly used in literature on particle transport (e.g.,~\cite{haller2008inertial}) is that for small enough $\eps$ the flow given by \eqref{e:MR} in $ {\mathbb R} ^4 $   can be reduced to a small perturbation (of order $\eps$) of the original  flow in $ {\mathbb R} ^2 $:
$\dot{\bf x} = {\bf v}(\bf{x}) + \eps {\bf f}(\bf{x}, \eps)$.
This application of Fenichel's results~\cite{fenichel} on normal hyperbolicity is done by Rubin, Jones, and Maxey~\cite{rubin1995settling} (see also the book~\cite{jones1995geometric}) for cellular flow the case when gravity is aligned with cell orientation ($a = 0$), but can be done in the same way for any ${\bf v}$.
To that end, let us rewrite~\eqref{e:MR} as a first order system:
\begin{equation} \label{e:MR-x-y}
 \dot{\bf x} = {\bf y}, \qquad \dot {\bf y} = - \frac{1}{\eps} ( {\bf y} - {\bf v} ( {\bf x} ) ),
 \qquad ({\bf x}, {\bf y}) \in \mathbb R^4.
\end{equation}
Applied to this problem, Fenichel's results imply the following proposition. We give just a quick sketch of the proof in Appendix~\ref{s:Fenichel-details} following the book~\cite{jones1995geometric} where a full proof can be found.
\begin{proposition} \label{p:fenichel}
  For any $M > 0$ there exists $\eps_0$ such that for any $a$, $b$, and $\eps$ with $|a|, |b| < M$ and $|\eps| < \eps_0$:
  \begin{enumerate}
      \item  The system~\eqref{e:MR-x-y},~\eqref{e:fluid} has an invariant manifold $\mathcal M_\eps$. It is globally attracting when $\eps>0$.
      For any $r \in \mathbb N$,  $\mathcal M_\eps$ is $O(\eps)$-close in $C^r$ to the manifold $\mathcal M_0$ given by ${\bf y} = {\bf v}({\bf x})$.

      \item The phase space $\mathbb R^4$ is foliated by stable fibers of the points of $\mathcal M_\eps$.
      Those fibers are two-dimensional manifolds, they are $C^r$ for any $r$, and for any point $q \in \mathbb R^4$ lying on the stable fiber of $p \in \mathcal M_\eps$, its forward orbit exponentially converges to the orbit of $p$.

      \item The dynamics on the invariant manifold $\mathcal M_\eps$
       can be written using ${\bf x}$ as a coordinate as
       \begin{equation} \label{e:perturbed}
          \dot {\bf x} = {\bf v} (\bf x) + \eps {\bf f}(\bf x, \eps),
       \end{equation}
        where ${\bf f} = {\bf f}({\bf x}, a, b, \eps)$ is $C^r$ for any\footnote{The required smallness of $\eps$ depends on $r$.} $r$ and periodic in ${\bf x}$.
        Here ${\bf f}({\bf x})$ is some non-Hamiltonian vector field with the leading terms given by ${\bf f} = - D{\bf v} \cdot {\bf v} + O(\eps)$, where $D{\bf v}$ is the Jacobian matrix of ${\bf v} (\bf x)$.
  \end{enumerate}
\end{proposition}

\noindent  It should be noted that the leading term $- D{\bf v} \cdot {\bf v}$   of ${\bf f}$ is negative convective derivative of $ {\bf v} $, i.e. minus the acceleration of the trajectory of $ {\bf v} $.
If $ {\bf v} $ happens to be a Hamiltonian vector field (as in our case), then
$$
    {\bf f} = (H_{yy}H_x - H_{xy}H_y, \; H_{xx}H_y - H_{xy}H_x ) + O(\eps)
$$
Curiously, the divergence of the negative of convective derivative of a Hamiltonian flow is twice the Hessian determinant of $ H $:
$$
     \Div {\bf f}  = 2\; {\rm det}\; H ^{\prime\prime}+O(\eps).
$$

\medskip

\textbf{Choice of a fundamental domain.} The vector field ${\bf v}(x, y)$ given by the Hamiltonian~\eqref{e:fluid} is $2\pi$-periodic in both $x$ and $y$, and is also invariant under the shift $(x, y) \mapsto (x+\pi, y+\pi)$; this also holds for the perturbation ${\bf f}$.
So, we can take the rectangle
$[-\frac \pi 2, \frac \pi 2] \times [-\frac \pi 2, \frac {3\pi} 2]$ with the sides identified by the shifts
\begin{equation} \label{e:shifts}
 (x, y) \mapsto (x+\pi, y+\pi),
 \qquad
 (x, y) \mapsto (x, y+2\pi)
\end{equation}
as a fundamental domain for both the unperturbed Hamiltonian system $\dot {\bf x} = {\bf v}({\bf x})$ and the perturbed system~\eqref{e:perturbed}.
We will refer to the torus resulting from taking the factor over~\eqref{e:shifts} as $\mathbb T^2$.
This reduces the original four-dimensional system~\eqref{e:MR} to the flow~\eqref{e:perturbed} on the torus $\mathbb T^2$.
The flow in $ {\mathbb R} ^4 $ decomposes into the   product of the periodic flow~\eqref{e:perturbed} on $\mathbb R^2$ and a contraction in transversal $ {\mathbb R} ^2 $-fibers.

\subsection{\poincare map} \label{s:poinc}
\textbf{Choice of a transversal.}
Assuming $b>0$, the vertical lines $\tilde T_k = \{ (x, y) \in \mathbb R^2, \; x = \pi k - \frac \pi 2 \}$, $k \in \mathbb Z$ are transversal to the Hamiltonian flow by~\eqref{e:fluid-ode} since $\dot x=b$ there.
This transversality survives for all $\eps$ sufficiently small in the perturbed flow~\eqref{e:perturbed}.
We choose $\tilde T_k$ as Poincare sections for the flow on $\mathbb R^2$; on the torus they all correspond to the same circle that we denote $T$.
Let us use
\[
  z = \frac{y-x}{2\pi}
\]
as a coordinate on each $\tilde T_k$.
As $z$ is invariant under the shift $(x, y) \mapsto (x+\pi, y+\pi)$ identifying $\tilde T_k$ with $\tilde T_{k+1}$, it is a well-defined coordinate on $T$, and it turns $T$ into the \emph{standard} circle $\mathbb R / \mathbb Z = S^1$.
Let $P: S^1 \mapsto S^1$ denote the first return map to $T$.
 $P$ is an increasing degree one circle map since  different trajectories of~\eqref{e:perturbed} do not intersect.
However, $P$ is undefined at certain points (we will later see that there are two such points).
Now, assume $b > 0$. Then $P$ has a natural lift $\tilde P: \mathbb R \mapsto \mathbb R$ which is the \poincare map from $\tilde T_k$ to $\tilde T_{k+1}$ using $z$ as a coordinate on both transversals. It does not depend on $k$ as both $z$ and the flow are invariant by the shift $(x, y) \mapsto (x+\pi, y+\pi)$.

\begin{figure} [H]
 \centering \includegraphics{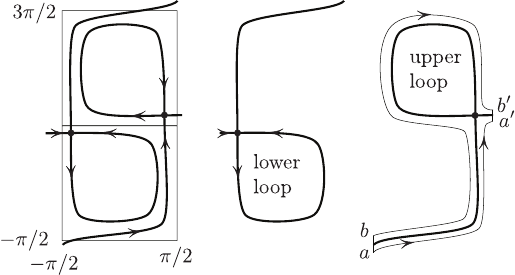}
 \caption{ The first return map $P_0$ of the Hamiltonian system is continuous: $ \lim_{|ab| \rightarrow 0 }|a ^\prime b ^\prime | = 0 $. }
 \label{fig:homoclinic}
\end{figure}

\textbf{Hamiltonian flow: the first return map is a pure rotation.}
  Here we briefly describe the unperturbed flow~\eqref{e:fluid} on the torus when $a, b \in (0, 0.1)$; proofs can be found in Appendix~\ref{a:saddles}. Let $P_0$ be the first return map for the Hamiltonian flow, and let $\tilde P_0$ be a lift of $P_0$ to the real line.
  This flow has two saddles, each having one homoclinic loop (Figure~\ref{fig:homoclinic}). The first return map $P_0$ is defined everywhere except for two points where stable manifolds of the saddles intersect the transversal $T$.
  A formula for $P_0$ arises from the fact that the Hamiltonian $H$ is preserved by the flow.
  Let us choose $H$ as the coordinate on $\tilde T_0$; since $\cos x \cos y = 0$ on $\tilde T$,   $ H(p_0) $ is a linear function of $z_0$.
  Let $ p_1\in \tilde T_1$ be the translate of $p_0 \in\tilde T_0$ under  $(x, y) \mapsto (x + \pi, y + \pi)$. Then
  $$
    \underbrace{H(p_0)}_{Z_0}  = \underbrace{H(p_1) + \pi a - \pi b}_{Z_1}.
  $$
  Since $H(p_0) = Z_0$ is already chosen as the coordinate of $p_0$, for consistency we must choose the right-hand side $Z_1$ to be the coordinate of $ p_1\in \tilde T_1$. Now if $\hat p_0\in \tilde T_0 $ and $\hat p_1 \in \tilde T_1 $ lie on the same trajectory then $H(\hat p_0)= H(\hat p_1)$ and hence their coordinates relate as
  $$
    \hat Z_1=\hat Z_0 + \pi a - \pi b,
  $$
  showing that the circle map is a rotation.
  Scaling by $2b \pi$ to return to $z$-coordinate gives an exact formula for this rotation:
  \begin{equation} \label{e:rotation}
   \tilde P_0: \mathbb R \mapsto \mathbb R,
   \quad
   z \mapsto z +    \frac{a-b}{2b} .
  \end{equation}
  To put this observation in a general context, note that Arnold showed~\cite{arnold1991topological} that under mild genericity conditions a general area-preserving flow on the 2-torus (i.e., a flow given by a multivalued Hamiltonian such as the one we consider) admits a transversal, and the first return map is undefined only at finitely many point and is a pure rotation on its domain. In our example a transversal exists a priory for all $a, b$ (except $a=b=0$), so no extra conditions are needed.

\medskip

\textbf{Perturbed flow: the first return map has two jumps.}
As $ \varepsilon $  increases from zero, each of the two homoclinic loops in Figure~\ref{fig:homoclinic} opens up as illustrated in Figure~\ref{f:homoclinic-open}, giving rise to a jump discontinuity, Figure~\ref{fig:poincare}. The jumps are the intervals on the transversal where trajectories escaping from former homoclinic loops merge into the flow ($I_1$ and $I_2$ in the figure on the right). These intervals are bounded by two branches of the unstable manifold of a saddle. The points of discontinuity of the return map (points $ b_1 $ and $ b_2$ in Figure~\ref{f:homoclinic-open}, right)  are  the intersections of the stable manifolds of the   saddles with the transversal.

\begin{figure} [H]
 \centering \includegraphics{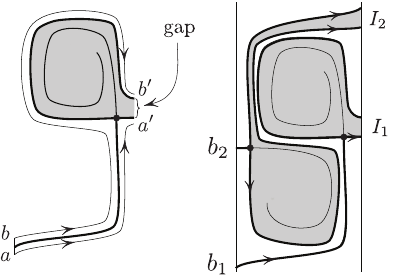}
 \caption{ When $\eps > 0$, the first return map $P$ has two jump discontinuities. For starting points $a$, $b$ separated by the stable manifold we have $ \lim_{|ab| \rightarrow 0 }|a ^\prime b ^\prime | > 0 $. }
 \label{f:homoclinic-open}
\end{figure}

\begin{figure} [H]
\centering
\begin{subfigure}{.5\textwidth}
 \centering
 \includegraphics[width=.5\linewidth]{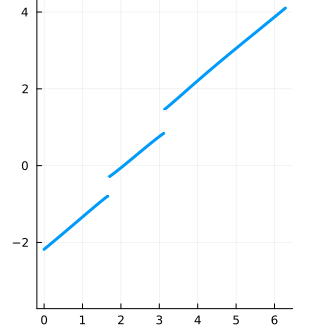}
 \caption{\poincare map $P$}
 \label{fig:poincare}
\end{subfigure}%
\begin{subfigure}{.5\textwidth}
 \centering
 \includegraphics[width=.5\linewidth]{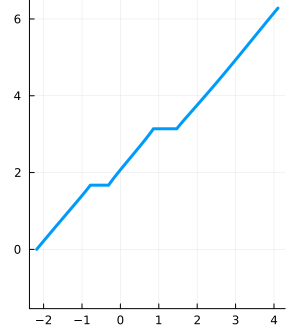}
 \caption{Inverse \poincare map $Q$}
 \label{fig:poincare-inv}
\end{subfigure}
\caption{Circle maps describing the particle dynamics}
\end{figure}
\subsection{Construction of the family of circle maps.} \label{s:Q}
In order to work with continuous maps, it will be convenient to consider the inverse \poincare map, Figure~\ref{fig:poincare-inv}.
The proposition below claims that fixing $b$ and $\eps$ and changing $a$ makes this inverse \poincare map a monotone family of circle maps with two flat spots. We actually need to take
$s = -a$ as the parameter to make the resulting family increasing rather than decreasing.

Assuming that the phase portrait in Figure~\ref{f:homoclinic-open} is valid, set the circle map $Q$ to be the inverse-time first return map of the flow~\eqref{e:perturbed} on $S^1 \setminus (I_1 \cup I_2)$, extending the map to the full circle by setting  $Q(I_1) = b_1$ and $Q(I_2) = b_2$.
The resulting map $Q$ is a continuous non-decreasing circle map.
Let $s=-a$ be the parameter and fix $b$ and $\eps$. This gives a family of circle maps $Q_s$.
\begin{proposition} \label{p:circle-maps-family}
  There exists $\gamma>0$ such that for any positive $\delta < \gamma$ there exists $\varepsilon_0 > 0$ such that
  \begin{enumerate}
    \item For any $a$ and $ b \in [\delta, \gamma]$ with $a \le 1.5b$ and any $\eps \in (0, \eps_0)$ the phase portrait of the flow~\eqref{e:perturbed} on the 2-torus $[-\frac \pi 2, \frac \pi 2] \times [-\frac \pi 2, \frac {3\pi} 2]$ factorized by~\eqref{e:shifts} is as in Figure~\ref{f:homoclinic-open}, more precisely:
    \begin{enumerate}
      \item This flow has two saddles, two sources, and no other fixed points.
      \item The forward trajectory of any point that is not on a stable mainfold of a saddle and is not one of the sources intersects $T$.
      \item The inverse-time first return map $Q$ is defined and continuous on the whole circle $T$ except for two closed intervals $I_1$, $I_2$. Each of these intervals is bounded by an intersection of the two unstable manifolds of a saddle with $T$. One-sided limits of this map at the endpoints of such interval $I_j$ exist and coincide with the point the points $b_j$ ($j=1, \ 2 $) where the stable manifold of the same saddle intersects $T$.
    \end{enumerate}

    \item  Fix $b \in [\delta, \gamma]$ and $\eps \in (0, \eps_0)$. Treat $s = -a$ as a parameter. Then the family
        $Q_s(z)$, $s \in [-\gamma, -\delta]$
      is a monotone family of circle maps with two flat spots (in the sense of Definition~\ref{d:monotone}).

      Moreover, the expansion constant $\lambda$ in the definition of a monotone family can be taken to be greater than $1 +c\eps$, where $c>0$ depends only on $\delta$.

  \end{enumerate}
\end{proposition}
\noindent This proposition is proved in Section~\ref{s:poinc-properties}.

\begin{remark} Since the value of the Hamiltonian of the unperturbed system changes by $O(\eps)$ while the trajectory starting on the transversal returns there (this is proved in the appendices, Remark~\ref{r:change-H-one-wind}), the \poincare map $P$ of the perturbed system is $O(\eps)$-close to the rigid rotation~\eqref{e:rotation} in $C^0$.
The argument with the change of the Hamiltonian works for the time-reversed flow as well, so the map $Q$ is also $O(\eps)$-close to a rotation in $C^0$, namely, to the inverse of the map~\eqref{e:rotation}.
\end{remark}

\medskip

\textbf{Rotation number and the drift slope.}
The next statement relates the rotation number of the first return map $P$ to the asymptotic slope of the trajectory: one is a linear function of the other.

\begin{remark} \label{r:rho-rho}
  For an initial point ${\bf x}_0 = (x_0, y_0) \in \tilde T$, let $z_0 = \frac{y_0-x_0}{2\pi} \in \mathbb R$;  suppose that $z_0$ has an infinite future $P$-orbit.
  Then
  \begin{enumerate}
    \item The rotation number
    \[
      \rho = \lim \limits_{n \to \infty} \frac{\tilde P^n(z_0) - z_0}{n}
    \]
    exists.
    \item   The forward trajectory ${\bf x}(t)$ of~\eqref{e:perturbed} starting at ${\bf x}_0$ is unbounded and at a finite distance from a ray with the slope $2 \rho + 1$.
  \end{enumerate}
\end{remark}
\noindent Apart from the particular formula connecting drift slope with rotation number, this is a well-known general fact on torus flows. A proof is given in  Appendix~\ref{s:rotation_number}.

\subsection{Proof of Theorem~\ref{t:main}.} \label{s:proof-main}
Now we are ready to prove Theorem~\ref{t:main} modulo Theorem~\ref{t:rotation-numbers}, propositions~\ref{p:fenichel} and~\ref{p:circle-maps-family}, and Remark~\ref{r:rho-rho}.
\begin{proof}[Proof of Theorem~\ref{t:main}]
  Let us first pick the constants in Theorem~\ref{t:main}. We take $\gamma$ as in Proposition~\ref{p:circle-maps-family} and take $\eps_0(\delta)$ to be the smaller of the two values:   $\eps_0(\delta)$ from Proposition~\ref{p:circle-maps-family} and   $\eps_0$ from Proposition~\ref{p:fenichel} with $M=\gamma$. For this choice of constants, we have conclusions of both propositions when
  $a, b \in [\delta, \gamma]$ and $\eps \in (0, \eps_0)$.

  By Proposition~\ref{p:fenichel}, the evolution of the particle's position ${\bf x}(t)$ governed by the 4-dimensional flow~\eqref{e:MR},~\eqref{e:fluid} is captured by the torus flow~\eqref{e:perturbed}, since  this invariant torus is globally attracting. The entire phase space is foliated by fibers with base points on the invariant torus; the flow is the product of contraction in the fiber and the flow on the torus $\mathcal M_\eps$.

  The phase portrait of the torus flow~\eqref{e:perturbed} is described by the first part of Proposition~\ref{p:circle-maps-family}. There are two sources, two saddles, and no other fixed points.
  Take a point $p \in {\cal M}_\eps \simeq \mathbb T^2$ that is not one of the sources and is not on the stable manifold (winding around the torus) of one of the saddles. The forward orbit of $p$ then intersects the transversal $T$, as claimed in Proposition~\ref{p:circle-maps-family}; repeating this argument shows that it intersects $T$ infinitely many times. According to Remark~\ref{r:rho-rho},  the lift of this forward trajectory to $\mathbb R^2$ is unbounded and at a finite distance of a ray with the slope $m = 2 \rho + 1$, where $\rho$ is the rotation number of this orbit under the first return map $\tilde P$. Recall that the circle map $Q$ defined in Section~\ref{s:Q} above denotes the inverse-time first return map continued to be constant on the two intervals where it is undefined, and
  $\tilde Q: \mathbb R \mapsto \mathbb R$ is its lift.
  Proposition~\ref{p:circle-maps-family} claims that it is a continuous non-decreasing degree one circle map (this is included in Definition~\ref{d:monotone}), so it has a rotation number $\rho(\tilde Q)$. Clearly, this implies that for any forward-unbounded trajectory $\rho(\tilde P) = - \rho(\tilde Q)$, and the particle drift slope is
  $m = 1 - 2 \rho(\tilde Q)$ for all initial data on the stable fiber of the points on $\mathcal M_\eps$ projecting to the point $p$ on the torus. So, the particle drift slope exists and equals $1 - 2 \rho(\tilde Q)$ for all initial data in $\mathbb R^4$ except for countably many two-dimensional manifolds (projecting to the sources on the torus) and countably many three-dimensional manifolds projecting to the stable manifolds of the saddles. These three-dimensional manifolds coincide with the stable manifolds of the saddles for the ODE in $\mathbb R^4$.

  We have checked that the particle drift slope $m$ is well defined for most initial data and does not depend on the initial point. Moreover, $m = 1 - 2 \rho(\tilde Q)$. Fixing $b$ and $\eps$ and parametrizing the map $Q$ by $s = -a$ gives rise to a circle map family $\tilde Q_s$. By Proposition~\ref{p:circle-maps-family}, $\tilde Q_s$ is a monotone family  (Definition~\ref{d:monotone}), and its rotation number as a function of $s$ satisfies the conclusions of Theorem~\ref{t:rotation-numbers}.
  So, the same properties also hold for the particle drift slope $m(\alpha)$ parametrized by $\alpha = a/b$. Indeed, $m(\alpha) = 1-2\rho(\tilde Q_s|_{s=-\alpha b})$. This gives the properties of the function $m(\alpha)$ claimed in Theorem~\ref{t:main}.
\end{proof}

\section{Circle maps with flat spots} \label{s:dim-H}
In this section we prove Theorem~\ref{t:rotation-numbers}.
Consider a monotone (in sense of Definition~\ref{d:monotone}) family $f_s$ of circle maps with $m > 0$ flat spots. Fix a lift $\tilde f_s: \mathbb R \mapsto \mathbb R$.
The following lemma summarizes some known results from  \cite{veerman1989irrational} and from references therein.
\begin{lemma} \label{l:rot-numbers}
\;

\begin{enumerate}
    \item Each $f_s$ has a well-defined and unique rotation number $\rho(s) = \rho(\tilde f_s)$.
    \item $\rho(s)$ depends continuously on the parameter $s$.
    \item $\rho(s)$ is non-decreasing in $s$ and strictly increasing when $\rho$ is irrational.
    \item When $\rho(s)$ is rational, it is constant on a non-degenerate closed interval   of $s$-values.
\end{enumerate}
\end{lemma}
\noindent For the sake of completeness and to deal with the fact that Definition~\ref{d:monotone} is a bit different from the class of circle maps families considered in~\cite{veerman1989irrational}, we provide a proof in Appendix~\ref{s:rotation_number}.

Lemma~\ref{l:rot-numbers} implies the first three conclusions of Theorem~\ref{t:rotation-numbers}. Let us now prove Conclusion~4 claiming that $\dim_H(C) = 0$. Conclusion~5 (steepness) will be a side result of this.
The key idea is to use the expansivity in a way similar to Veerman's proof for the case of one flat spot~\cite{veerman1989irrational}. This is done below in the proof of Lemma~\ref{l:expand}.
But then instead of estimating the lengths of intervals in the complement to $C$, as done by Boyd~\cite{boyd1985structure} and later Veerman, we write $C = \cap_N C_N$ and estimate the lengths of the intervals in $C_N$. This gives a shorter proof, as there is no need to prove that $C$ has zero measure prior to proving that it has zero Hausdorff dimension.

Let $\tilde I_j$ be lifts of the flat spots $I_j$, and let $\tilde b_j(s) = \tilde f_s(I_j(s))$.
For a positive integer $N$ let $C_N$ be the set of all parameters $s$ such that the rotation number is irrational or rational with the denominator greater than $mN$: 
\[
  C_N = \Big\{ s \in [s_{min}, s_{max}]: \rho(s) \not \in \mathbb Q \text{ or } \rho(s) =  \frac p q \in \mathbb Q \text{ with } \gcd(p, q) = 1, \; q > m N  \Big\}.
\]
Clearly, $C = \bigcap\limits_{N=1}^\infty C_N$.

The following lemma and its corollary are used to   estimate the Hausdorff dimension and the steepness of $\rho$.

Set\footnote{This generalizes the function $\frac 1 N \tilde f_s^{N-1}(\tilde b_1(s))$ Veerman considered for one flat spot.}
\begin{equation} \label{e:psi-circle}
    \psi_N(s) = \sum_{j=1}^m \tilde f_s^{N-1}(\tilde b_j(s)).
\end{equation}
\begin{lemma} \label{l:expand}
 Pick any $s_0 \in C_N$. Then the `slope' of $\psi(s)$ near $s_0$ is $\gtrsim \lambda^N$. More formally, there exists $\delta > 0$ such that for any $\tau_1,  \tau_2 \in [s_0 - \delta, s_0 + \delta]$ with $\tau_2 \ge s_0 \ge \tau_1$ we have
 \begin{equation} \label{e:slope}
    \psi(\tau_2) - \psi(\tau_1) \ge \nu \lambda^{N-1}(\tau_2 - \tau_1),
 \end{equation}
 where the constants $\lambda$ and $\nu$ are from Definition~\ref{d:monotone}: $\lambda > 1$ is the expansion constant, and $\nu > 0$ is a lower bound on $\frac{d \tilde b_j}{ds}$.
\end{lemma}
\begin{proof}
Set $\mathcal I(s) = \cup_{j=1}^m I_j(s)$.
First, let us show that for some $j$ we have
\begin{equation} \label{e:disjoint-0}
    f_{s_0}^n(b_j(s_0)) \not\in \mathcal I(s_0), \qquad
n=0, 1, \dots, N-1.
\end{equation}
Suppose the opposite holds, so that for any $j$ there exist $k_j \le m$ and $n_j \le N-1$ such that $f_{s_0}^{n_j}(b_j(s_0)) \in I_{k_j}(s_0)$. Let us write $j \mapsto k_j$ if this is the case. As there are only $m$ flat spots, by the pigeonhole principle we can find a cycle 
\[
    j_1 \mapsto j_2 \mapsto \dots \mapsto j_q \mapsto j_1, \qquad q \le m.
\]
As $j_s \mapsto j_{s+1}$ implies
$f_{s_0}^{n_{j_s} + 1} (b_{j_s}) = b_{j_{s+1}}$, this shows that $f_{s_0}$ has a periodic orbit with the period
\[
(n_{j_1} + 1) + \dots + (n_{j_{q}}+1) \le qN \le mN.
\]
Then, $\rho(s_0)$ is rational with the denominator $\le mN$, which contradicts $s_0 \in C_N$. This contradiction shows that~\eqref{e:disjoint-0} holds for some $j$.

Now, fix $j$ such that we have~\eqref{e:disjoint-0}. By the continuity of $I_j(s)$, for some $\delta > 0$ for any $s \in [s_0 - \delta, s_0 + \delta]$ we also have~\eqref{e:disjoint-0} with $b_j(s)$ instead of $b_j(s_0)$ (while we still iterate $f_{s_0}$):
\begin{equation} \label{e:disjoint}
    f_{s_0}^n(b_j(s)) \not\in \mathcal I(s_0), \qquad
n=0, 1, \dots, N-1.
\end{equation}
Let us now fix $\tau_1,  \tau_2 \in [s_0 - \delta, s_0 + \delta]$ with $\tau_1 \le s_0 \le \tau_2$ and prove~\eqref{e:slope}.
As all the summands in~\eqref{e:psi-circle} are monotone with respect to $s$, it is enough to prove that for $j$ picked above
\begin{equation}
    \tilde f_{\tau_2}^{N-1}(\tilde b_j(\tau_2))
    -
    \tilde f_{\tau_1}^{N-1}(\tilde b_j(\tau_1))
    \ge
    \nu \lambda^{N-1}(\tau_2 - \tau_1).
\end{equation}
This is done using expansivity as in Veerman's lemmas~5.1 and~5.2. We have $\tilde b_j(\tau_2) - \tilde b_j(\tau_1) \ge \nu(\tau_2 - \tau_1)$ due to property~\ref{prop:flat-spots} of monotone families. By monotonicity, we have
\[
    \tilde f^{N-1}_{\tau_2}(\tilde b_j(\tau_2))
    -
    \tilde f^{N-1}_{\tau_1}(\tilde b_j(\tau_1))
    \ge
    \tilde f^{N-1}_{s_0}(\tilde b_j(\tau_2))
    -
    \tilde f^{N-1}_{s_0}(\tilde b_j(\tau_1)).
\]
For any $n \le N-1$ the arc $[f^n_{s_0}(\tilde b_j(\tau_1)), f^n_{s_0}(\tilde b_j(\tau_2))]$ does not intersect $\mathcal I(s_0)$. Indeed, by the intermediate value theorem any $x$ in this arc can be written as
$x = f^n_{s_0}(\tilde b_j(s)) \not \in \mathcal I(s_0)$,
with $s \in [\tau_1, \tau_2]$.
Thus, due to the expansivity outside of $\mathcal I$ we get
\[
    \tilde f^{N-1}_{s_0}(\tilde b_j(\tau_2))
    -
    \tilde f^{N-1}_{s_0}(\tilde b_j(\tau_1))
    \ge
    \lambda^{N-1}(\tilde b_j(\tau_2) - \tilde b_j(\tau_1))
    \ge
    \nu \lambda^{N-1}(\tau_2 - \tau_1).
\]
\end{proof}

\begin{corollary} \label{c:dim-H-key}
  Any interval $[s_1, s_2] \subset C_N$ satisfies
  \begin{equation}
    s_2 - s_1 \le K N \lambda^{-N}
  \end{equation}
  for some constant $K$ independent of $N$.
\end{corollary}
\begin{proof}
  We will construct a sequence $s_1 = y_0 < y_1 < \dots < y_L = s_2$ so that we can estimate $\psi(y_{i+1}) - \psi(y_i)$ using Lemma~\ref{l:expand}.
  To this end, cover each $s \in [s_1, s_2]$ by its $\delta$-neighborhood provided by Lemma~\ref{l:expand}. Select a finite subcover $\{ \tilde U_i \}$ and remove all intervals contained inside another interval. Let $\{ U_i\} $, $i=1, \dots, L$ denote the resulting subcover reordered so that the centers form an increasing sequence. Then, $U_i$ intersects $U_{i+1}$ for any $i$ (otherwise, the interval $U_j$ covering the empty space between them will cover $U_i$ if $j<i$ or $U_{i+1}$ if $j>i+1$).
  Set $y_0 = s_1$.
  When $i=1, \dots, L-1$, set $y_i$ to be any point in $U_i \cap U_{i+1}$ lying between the centers of $U_i$ and $U_{i+1}$. Finally, set $y_L = s_2$.
  By construction, $[y_i, y_{i+1}]$ is covered by $U_{i+1}$ for any $i=0, \dots, L-1$ and contains the center of $U_{i+1}$. Thus, Lemma~\ref{l:expand} implies $\psi(y_{i+1}) - \psi(y_i) \ge \nu \lambda^{N-1}(y_{i+1} - y_i)$. Adding these inequalities together yields
  \begin{equation*}
   \psi(s_2) - \psi(s_1) \ge \nu \lambda^{N-1} (s_2 - s_1).
 \end{equation*}
  As $f$ has degree one, there is $M > 0$ such that for any $N \ge 1$ and any $j$ we have $|\tilde f_s^N(\tilde b_j)| < MN$ for all $s$. This implies $|\psi(s)| \le MmN$ and $\psi(s_2) - \psi(s_1) \le 2MmN$. Combining with the previous inequality yields the estimate on $s_2 - s_1$.
\end{proof}

Now, we are ready to prove steepness. For right endpoints, steepness follows from the corollary below. For left endpoints, it can be proved in the same way.
\begin{corollary} \label{c:steep}
  Let $s_0$ be a right endpoint of some rational interval $\{ s: \rho(s) = \frac p q\}$. Then, for all small enough $s > 0$ we have
  \begin{equation}
    \rho(s_0 + s) - \rho(s_0) >\frac{D}{ q|\ln s|}, \ \hbox{where} \ \  D= \frac{\ln \lambda}{4 m}.
  \end{equation}
\end{corollary}
\begin{proof}
  By Corollary~\ref{c:dim-H-key} for large enough $N$ the length of each interval contained in $C_N$ is less than $\lambda^{-N/2}$. Assuming $s$ is small enough, take $N = 2 \big[ \frac{-\ln s}{\ln \lambda} \big] + 2$. Then $N \ge -2 \log_{\lambda} s$, and $\lambda^{-N/2} \le s$.
  Then $(s_0, s_0+s)$ cannot be a subset of $C_N$, and so there is $s' \in (s_0, s_0 + s)$ with $\rho(s')$ being rational with the denominator $Q \le mN$. Hence, $\rho(s') - \rho(s_0) \ge \frac 1 {qQ} \ge \frac 1 {qm N} $.

  Provided that $s$ is small enough, we have
  $N \le 4 \frac{|\ln s|}{\ln \lambda}$
  and thus
  $\rho(s') - \rho(s_0) \ge \frac {\ln \lambda} {4qm|\ln s|}$, as claimed.
\end{proof}

Before we estimate the Hausdorff dimension of $C$, let us recall the definition. For $d > 0$ and a finite or countable union $\mathcal U$ of closed intervals $U_i$ denote $m_d = \sum_i |U_i|^d$ and
$\diam \; \mathcal U = \sup_i |U_i|$.
The $d$-dimensional \emph{Hausdorff measure} of a set $X$ is defined as
\[
    \mathcal H^d(X) = \sup_{\delta > 0} \; \inf \{ m_d(\mathcal U):
    \diam \; \mathcal U < \delta \text{ and } \mathcal U \text{ covers } X \},
\]
and the \emph{Hausdorff dimension} of $X$ is
\[
    \dim_H(X) = \inf \{ d \ge 0: \mathcal H^d(X) = 0\}.
\]
Note that $C_N$ is a finite union of intervals, and let $\mathcal U_N$ be the union of the closures of these intervals.
\begin{lemma} \label{l:cover}
 \[
    m_d(\mathcal U_N) = O(N^{2+d}) (\lambda^d)^{-N},
    \qquad
    \diam \; \mathcal U_N = O(N \lambda^{-N}),
 \]
 where the constants in the $O$-estimates do not depend on $N$.
\end{lemma}

\begin{proof}
  The estimate on $\diam \; \mathcal U_N$ follows from Corollary~\ref{c:dim-H-key}. The number of intervals in $\mathcal U_N$ is $O(N^2)$, as there are $O(N^2)$ rational numbers with the denominator $ \le mN$ in any fixed interval. Together with the estimate on $\diam \; \mathcal U_N$ this gives the estimate on $m_d(\mathcal U_N)$.
\end{proof}
This lemma implies that for any $d > 0$, when $N \to\infty$, we have
$m_d(C_N) \to 0$ and $\diam \; C_N \to 0$.
Thus, $\mathcal H^d(C) = 0$, and $\dim_H(C) = 0$. This completes the proof of Theorem~\ref{t:rotation-numbers}.

\section{Proof of monotonicity of the circle map family} \label{s:poinc-properties}
The first part of Proposition~\ref{p:circle-maps-family} is proved in Appendix~\ref{s:phase-portrait}.
Here we prove the second part of the proposition using the first part as well as several technical lemmas proved below in appendices.
We need to verify that for fixed $b$ and $\eps$ the family $Q_s(z)$ defined in Section~\ref{s:Q} satisfies properties 1-5 in Definition~\ref{d:monotone}, treating $s=-a$ as the parameter.
\medskip

\noindent \textbf{1a. Continuity:} \emph{$Q_s(z)$ is continuous with respect to both $z$ and $s$.} The continuity in $z$ is stated in conclusion~(c) of the first part of Proposition~\ref{p:circle-maps-family}. The continuity with respect to $s$ when $z$ is outside the flat spots follows from continuous dependence of solutions of the ODE~\eqref{e:perturbed} with reversed time with respect to the parameter $s$.
When $z$ is inside a flat spot, continuity in $s$ can be deduced from continuity in $s$ outside the flat spots and monotonicity and continuity in $z$.
Indeed, take any value $s_0$ of the parameter and any point $z \in \mathbb R$ lying in a flat spot (more precisely, in its lift to $\mathbb R$). Let $J \subset \mathbb R$ be a connected component of the lift of a flat spot that covers $z$. Set $z' = \tilde Q_{s_0}(z)$. Take two points $a, b \in \mathbb R$ outside flat spots so that for $s=s_0$ the interval $[a, b]$ covers $J$ but $a$ and $b$ are so close the two endpoints of $J$ that $|\tilde Q_{s_0}(a) - z'| < \eps/2$ and $|\tilde Q_{s_0}(b) - z'| < \eps/2$.
By the continuity in $s$ outside flat spots, for any $s_1$ close enough to $s_0$ we have
$|\tilde Q_{s_1}(a) - z'| < \eps$ and $|\tilde Q_{s_1}(b) - z'| < \eps$.
As $\tilde Q_{s_1}(z)$ is between $\tilde Q_{s_1}(a)$ and $\tilde Q_{s_1}(b)$, this implies the continuity of $Q_s(z)$ at $s_0$.

\medskip

\noindent \textbf{1b. Monotonicity $\tilde Q_s(z)$ in $ z $ } follows from the fact that trajectories of~\eqref{e:perturbed} do not intersect each other.

\medskip
\noindent Monotonicity with respect to the parameter $s$ is much less obvious, and it is helpful to establish the other properties first so that they can be used to prove monotonicity in $s$. This is done below under the name ``property 1c''.

\medskip

\noindent \textbf{2. Degree one:} \emph{$Q_s$ has degree 1 for all $s$.}
As the vector field~\eqref{e:perturbed} is periodic and solutions of the corresponding ODE are unique, two solutions starting at the points $(x_0, y_0)$ and $(x_0, y_0 + 2 \pi)$ are related by the vertical shift by $2 \pi$. Hence, for any two points on the transversal (which is vertical) their images under the inverse-time \poincare map $\tilde Q_s$ also differ by the same shift.

\medskip

\noindent \textbf{3. Flat spots:} \emph{there are two flat spots for each $s$. More precisely, there are two disjoint closed arcs $I_1(s)$ and $I_2(s)$ such that $Q_s$ is constant on each arc.   The endpoints of these arcs depend on $s$ continuously.} Existence of two flat spots follows from the construction of $Q_s$ described in Section~\ref{s:poinc}. Indeed, by the first part of Proposition~\ref{p:circle-maps-family} there are two intervals $I_1$ and $I_2$ bounded by saddle separatrices (Figure~\ref{fig:homoclinic}) such that outside these intervals time-reversed \poincare map is defined, and we set $Q_s$ to be constant on these intervals. The continuity of the endpoints of $I_1(s)$ and $I_2(s)$ with respect to~$s$ follows from the continuous dependence of the separatrices on $s$ and their transversality to the \poincare section.

\medskip

\noindent \textbf{4. Expansivity:} \emph{
There exists $\lambda > 1$ such that for each $s$ the map $Q_s$ is $\lambda$-expanding outside $I_1 \cup I_2$. Namely, if
$z \not \in I_1 \cup I_2$ and $z \in \mathbb R$ is any lift of $z$,
$\tilde f_s(z)$ is differentiable in $z$ and $\frac{d}{dz}(\tilde f_s(z)) > \lambda$. Moreover, we can take $\lambda > 1 + c \eps$ with $c > 0$ depending only on the constant $\delta$ bounding from below $a$ and $b$ (recall the assumption $a, b > \delta > 0$ in Proposition~\ref{p:circle-maps-family}).
}

This follows from the following lemma proved below in Appendix~\ref{s:expand}.
\begin{lemma} \label{l:contract}
  There exists $\gamma > 0$ such that for any positive $\delta < \gamma$ there exist $\eps_0 > 0$ such that for any $a, b \in (\delta, \gamma)$ and for any positive $\eps < \eps_0$ the map $P$ is contracting except for the two points where it has a jump discontinuity (Figure~\ref{f:homoclinic-open}):
  \[
    \frac{d}{dz} P(z) < 1 - \eps.
  \]
\end{lemma}
\noindent As $Q$ is inverse to $P$ outside the flat spots, this implies expansivity for $Q$:
\[
  \frac{d}{dz} Q = \Big( \frac{d}{dz} P \Big)^{-1} > (1 - \eps)^{-1} > 1 + \eps.
\]

\medskip

\noindent \textbf{5. Flat spots heights increase with the parameter:}
\emph{ $\frac{d}{ds} \tilde b_j(s) > 0.15$ for all $s$, $j=1, 2$.} As $s=-a$, this follows from the following lemma proved below in Appendix~\ref{a:plateau-heights}.
\begin{lemma} \label{l:plateau-heights}
  For any positive $\delta < 0.05$ there exists $\eps_0 > 0$ such that for any $a, b \in (\delta, \; 0.1-\delta)$ and any positive $\eps < \eps_0$,
  \[
    \frac{\partial}{\partial a} \tilde b_j(a, b, \eps) < -0.15.
  \]
\end{lemma}

\medskip

\noindent \textbf{1c. Monotonicity in the parameter:} \emph{$\tilde Q_s(z)$ is non-decreasing with respect to $s$.}
We need the following lemma, proved in Appendix~\ref{a:monotone}. Recall that $\tilde P: \mathbb R \mapsto \mathbb R$ denotes the forward-time \poincare map (Section~\ref{s:poinc}); $\tilde P =\tilde P_{a, b, \eps} $  depends on the parameters $a$, $b$, and $\eps$.

\begin{lemma} \label{l:monotone}
  The exists $\gamma > 0$ such that for any positive $\delta < \gamma$ there exists $\eps_0 > 0$ such that for any $a, b \in (\delta, \gamma)$ and for any positive $\eps < \eps_0$,
 \[
 \frac{\partial}{\partial a} \tilde P(z) > 0,
 \]
 provided that $z \in \mathbb R$ is such that $\tilde P(z)= \tilde P_{a, b, \eps}(z)$ is defined.
\end{lemma}
\noindent Now we can prove monotonicity of $\tilde Q(z)$ with respect to $s=-a$. If $z$ is outside the flat spots, it follows from Lemma~\ref{l:monotone}, as $\tilde Q$ is (locally) the inverse of $\tilde P$. By continuity, this also implies that $\tilde Q$ is non-decreasing in $s$ if $z$ is on the boundary of a flat spot. Finally, monotonicity for $z$   strictly inside a flat spot follows from the monotonicity of flat spots' heights discussed above. \qed

\newpage

\begin{appendices}

\section{Incorporating external force in the fluid velocity field} \label{a:gravity}
In this appendix we show how an adding an external force acting on a particle carried by fluid flow results in the model~\eqref{e:MR},~\eqref{e:fluid}  we consider.  This model was proposed by Maxey and Corrsin~\cite{maxey1986gravitational}, and the discussion below follows their paper.

Let ${\bf u}({\bf x})$ be the fluid velocity at the point ${\bf x}$, and let the constant vector ${\bf g}$ be the free fall acceleration.
A simplified form of Maxey-Riley equation is Newton's second law for the position ${\bf x}(t)$ of a spherical aerosol particle carried by the fluid and pulled by gravity:
\[
  m\ddot {\bf x} = -k (\dot {\bf x} - {\bf u}({\bf x})) + m{\bf g},
\]
where $m$ is the particle mass and $k$ is the drag coefficient; for spherical particles
$k=6 \pi a \mu$ where $a$ is the radius and $\mu$ is the fluid viscosity.
Setting ${\bf w} = m{\bf g}/k$ (this vector can be interpreted as the terminal velocity of the particle in a still fluid) and letting  ${\bf v} = {\bf u} + {\bf w}$   gives
\[
  m\ddot {\bf x} = -k (\dot {\bf x} - {\bf v}({\bf x})).
\]
Finally, dividing    by $m$ and setting $\eps = m/k = m/(6 \pi a \mu)$ yields  the model~\eqref{e:MR},~\eqref{e:fluid}:
\[
\ddot {\bf x} = - \frac{1}{\eps} ( \dot {\bf x} - {\bf v} ( {\bf x} ) ),
\]
where
\[
    {\bf v} = {\bf u} + {\bf w},
    \qquad \eps = m/(6 \pi a \mu),
    \qquad {\bf w} = m{\bf g}/(6 \pi a \mu).
\]
This shows that adding external force is equivalent to modifying the ``carrying'' vector field.

\section{Properties of the Hamiltonian flow} \label{a:saddles}
In this appendix, we establish several properties of the Hamiltonian flow~${\bf v}$,~\eqref{e:fluid} and~\eqref{e:fluid-ode}. For brevity, we will refer to the two branches of the invariant manifold of a saddle as \emph{separatrices} throughout this appendix. Accordingly, in our notation, each saddle has two stable separatrices and two unstable ones.

Consider the grid
$\Gamma = \{(k_1 \pi + \frac \pi 2, k_2 \pi + \frac \pi 2), \; k_1, k_2 \in \mathbb Z\}$ formed by the saddles of this flow when $a=b=0$ so that   $H = \cos x \cos y$.
When $a$ and $b$ are small, the actual saddles are close to the nodes of $\Gamma$, and we index these saddles by the nearest nodes of $\Gamma$. We shall say that a saddle of~${\bf v}$    is \emph{odd} if $k_1 + k_2$ for the nearest node of $\Gamma$ is odd, and \emph{even} otherwise.

\begin{lemma} \label{l:Hamilt-at-saddles}
    Suppose $0 \le a, b \le 0.1$.
    Then, the Hamiltonian~\eqref{e:fluid} at a saddle of ${\bf v}$ is
    \begin{equation}\label{e:HG}
        H =  b y_\Gamma - a x_\Gamma \pm K
    \end{equation}
    with the plus sign if the saddle is odd and the minus sign if it is even.
    Here $x_\Gamma, y_\Gamma$ are the coordinates of the nearest node of $\Gamma$, and
    \begin{equation} \label{e:K}
    K = \frac 1 2 \Big( f(b-a) - f(b+a) \Big)
    \qquad
    \text{with }
    f(x) = \sqrt{1-x^2} + x \arcsin x.
    \end{equation}
\end{lemma}
\begin{proof}
It is convenient to rewrite
\[
    H = \frac{1} 2 \Big(\cos(x+y) + \cos(x-y)\Big) + by - ax
\]
and use $X=x+y$, $Y=x-y$ as new variables\footnote{This is not a canonical transformation, but we will not write the Hamilton equations in the new coordinates; we just want to find the saddle points of the function $H$.}. Then, $x=\frac{X+Y} 2$, $y=\frac{X-Y} 2$, and
\begin{equation} \label{e:appB-H-X-Y}
  H = \frac 1 2 \Big( \cos X + \cos Y + (b-a) X - (a+b) Y \Big).
\end{equation}
A straightforward computation yields   the coordinates of the odd saddles:
\[
    X = \arcsin(b-a) + 2\pi k,
    \qquad
    Y = \pi + \arcsin(a+b) + 2 \pi l
\]
and of the even saddles:
\[
    X = \pi - \arcsin(b-a) + 2\pi k,
    \qquad
    Y = -\arcsin(a+b) + 2 \pi l.
\]
Substitution   into~\eqref{e:appB-H-X-Y} yields ~\eqref{e:HG}.
\end{proof}

\medskip

\textbf{``Chess'' game rule -- the proof.}
Here we prove that the forward-unbounded  trajectories of~${\bf v}$ wind their way according   the combinatorial ``Chess game'' rule stated in Section~\ref{s:Ham}.
It is clear that when $a$ and $b$ are small, unbounded trajectories are close to the edges of $\Gamma$ and turn near its vertices.
Trajectories turning left and right are separated by the stable manifold of the nearby saddle of~${\bf v}$; solutions precisely on the separatrix are bounded and thus are not considered. Consider an unbounded trajectory making a turn near a saddle. Let $H_0$ be the value of $H$ on this trajectory, and let $H_{s}$ be the value of $H$ on the stable manifold which this trajectory follows before turning left or right as after approaching the saddle with the same value $ H_s$ of the Hamiltonian. And the turn direction is determined by the sign of $H_s - H_0$.
The values of $H_s$ are computed above in Lemma~\ref{l:Hamilt-at-saddles}, and the answer was different for odd and even saddles, this is the reason for the chess coloring used in the statement of the ``Chess'' rule.
For an odd saddle near a node $(x_\Gamma, y_\Gamma)$, we have $H_s = b y_\Gamma - a x_\Gamma + K$, so that a left turn happens if $b y_\Gamma - a x_\Gamma + K < H_0$, while a right turn happens if $b y_\Gamma - a x_\Gamma + K > H_0$.
The equation of the line corresponding to odd vertices is thus $b y - a x = c_o$ with $c_o=H_0 - K$; the turn is determined by the position of the grid node relative to this line. For even vertices, the same considerations work, but give another line: $b y - a x = c_e$ with $c_e = H_0 + K$.
\qed

\medskip

\begin{lemma} \label{l:ham-sep}
  For any $a, b \in (0, 0.1]$ the Hamiltonian flow ${\bf v}$ has one saddle and one center in each \emph{cell} $[-\frac \pi 2 + \pi k, \frac \pi 2 + \pi k] \times [-\frac \pi 2 + \pi l, \frac \pi 2 + \pi l]$ with $k, l \in \mathbb Z$. Two separatrices of the saddle form a homoclinic loop fully contained inside the cell, while two other separatrices do not form a loop (we shall call them \emph{free} separatrices).

  If, additionally, $a \le 1.5b$, then the free stable separatrix of the saddle intersects the transversal $x = - \frac \pi 2 + \pi k$ and the free unstable separatrix intersects the transversal $x = \frac \pi 2 + \pi k$.
  Moreover, the \poincare map $P_0$ on the torus (Section~\ref{s:poinc}) is defined at all points on the transversal other than the two points where free stable separatrices of the two saddles intersect the transversal.
\end{lemma}
\begin{proof}
  \;
  \begin{itemize}
    \item
    \emph{In each cell there are exactly two fixed points, one center and one saddle.}
    This is a straightforward computation that can be done as in the proof of Lemma~\ref{l:Hamilt-at-saddles}.

    \item \emph{One pair of separatrices of the saddle forms a homoclinic loop that lies inside the cell.}
    The cells we consider are the pieces $\mathbb R^2$ is divided into by a family of vertical lines and a family of horizontal lines.
    Restricted to the vertical lines $x = \frac \pi 2 + \pi k$, the horizontal component of ${\bf v}$ on these lines is equal to $b$ by~\eqref{e:fluid-ode}.
    Vector field $ {\bf v} $ is transversal to the
    vertical lines $x = \frac \pi 2 + \pi k$ since
    $ \dot x = b>0 $ there.
    Hence, these lines are transversal to ${\bf v}$, and all trajectories cross them all in one direction, from left to right. A similar statement holds for the horizontal lines $y = \frac \pi 2 + \pi k$ where   $\dot y =a$. As each of these lines can only be crossed in one direction, periodic trajectories and homoclinic loops are only possible inside the cells.

    The center is surrounded by periodic trajectories; the boundary of the domain formed by these trajectories must contain a critical point of $H$. This critical point is a fixed point of ${\bf v}$, so it must be the saddle. This means that the boundary of this domain is a homoclinic loop.

    \item \emph{The other two pairs of separatrices does not form   loops.}

    If two other separatrices form another loop, one can find two periodic orbits near the two loops such that domains bounded by them are disjoint and do not contain the saddle. It is proved by considering two cases: if the interiors of the loops are disjoint and if the interior of one loop contains the other loop. In each domain bounded by this two periodic orbits there is a fixed point by \PH index theorem, which is a contradiction: there are only two fixed points, and the saddle is outside these domains by construction.

    \item Suppose that $a \le 1.5b$. Then, \emph{in any vertical strip $-\frac \pi 2 + \pi k \le x \le \frac \pi 2 + \pi k$, $k \in \mathbb Z$, all saddles have different values of the Hamiltonian $H$.}
    We can assume WLOG that $k=0$, and the strip is $-\frac \pi 2 \le x \le \frac \pi 2$. Denote this strip by $\Pi$.
    A straightforward computation shows that saddles in $\Pi$ come in horizontal pairs, and different pairs are shifted vertically by $2 \pi l$. These saddles are near the grid nodes $(-\frac \pi 2, \frac \pi 2 + 2 \pi l)$ and $(\frac \pi 2, \frac \pi 2 + 2 \pi l)$, where $l \in \mathbb Z$.

    By Lemma~\ref{l:Hamilt-at-saddles}, the values of $H$ at these saddles are $\frac \pi 2 (4b l + b \pm a) \mp K$ with "$-a$" for the left saddle in each pair and "$+a$" for the right one.  As $K$ is given by~\eqref{e:K} and the function $f$ there satisfies $f(0)=1$ and $|f'|=|\arcsin x| < 0.3$ on $[-0.2, 0.2]$, we have $|K| < 0.3a$.
    Hence, the values of $H$ at saddles with different $l$
    differ by at least $\frac \pi 2 (4b-2a) - 2|K|$, which is positive as
    \[
      \frac \pi 2 (4b-2a) > 4b-2a \ge 4b - 2\cdot 1.5 b \ge b \ge 0.6a  > 2|K|.
    \]
    The values of $H$ at two saddles with the same $l$ differ by at least $\frac \pi 2(2a)-2|K| > 0$.

    \item Suppose that $a \le 1.5b$. Then, \emph{free stable separatrix of the saddle intersects the transversal $x = - \frac \pi 2 + \pi k$ and free unstable separatrix intersects the transversal $x = \frac \pi 2 + \pi k$.}

    Consider the vertical strip $-\frac \pi 2 + \pi k \le x \le \frac \pi 2 + \pi k$ containing the saddle and a rectangle $R$ given by $-\frac \pi 2 + \pi k \le x \le \frac \pi 2 + \pi k$ and $|y| \le M$, where the number $M$ is so large that $H$ restricted to the top side of $R$ is greater than at the saddle, and $H$ restricted to the bottom side of $R$ is less than at the saddle. As ${\bf v}$
    goes right at the left and right sides of $R$, the free unstable separatrix can leave $R$ only via its right side $x = \frac \pi 2 + \pi k$, and the free stable separatrix can leave $R$ only via its left side $x = -\frac \pi 2 + \pi k$.

    It remains to prove that free separatrices leave $R$, for definiteness let us prove that the free unstable separatrix cannot stay in $R$. Let $p$ be any point on this separatrix inside $R$ and let $\omega(p)$ be its $\omega$-limit set. If the separatrix stays in $R$, we have $\omega(p) \subset R$. By \PB theorem, $\omega(p)$ can be either (1) a fixed point, (2) a periodic orbit, or (3) a union of fixed points and homoclinic or heteroclinic connections between them. As the sources are surrounded by periodic orbits, they cannot belong to $\omega(p)$. The only saddle that can belong to $\omega(p)$ is the "original" saddle whose separatrix we consider, as at other saddles the value of $H$ is separated from the value of $H$ on the orbit of $p$. However, bacause $H$ is preserved at the orbit of $p$, it can only return near the original saddle along a separatrix. This would imply a homoclinic loop, which is not the case as the separatrix we consider is free. Finally, if $\omega(p)$ is a periodic orbit, the backward trajectory of $p$ stays at these periodic orbit, while the backward trajectory of $p$ actually approaches the saddle as $p$ is on an unstable separatrix. This exhausts all possibilities. The resulting contradiction shows that the free unstable separatrix leaves $R$ and so intersects the transversal $x = \frac \pi 2 + \pi k$.

    \item  \emph{The \poincare map $P_0$ on the torus (Section~\ref{s:poinc}) is defined at all points on the transversal other than the two points where stable separatrices of the two saddles intersect the transversal. }
    Consider the \poincare map $\tilde P_0$ on $\mathbb R^2$ from the transversal $x = -\frac \pi 2$ to the transversal $x = \frac \pi 2$. We will prove that it is defined at all points other than the intersections of stable separatrices of the saddles in the strip $\Pi$ given by $- \frac \pi 2 \le x \le \frac \pi 2$ with the transversal $x = -\frac \pi 2$. This will imply the statement about $P_0$ as $\tilde P_0$ projects to $P_0$.

    Consider a point $p$ on the transversal $x = -\frac \pi 2$ and suppose its orbit stays in the strip $\Pi$.
    Let us use \PB theorem as above to get a contradiction. The set $\omega(p)$ cannot be a periodic orbit as the orbit of $p$ cannot return to the transversal $x = -\frac \pi 2$. It cannot contain a source as those are surrounded by periodic orbits. Finally, it cannot contain a saddle as $H(p)$ is different from the values of $H$ at the saddle. Hence, the \poincare map is defined at $p$.
  \end{itemize}

\end{proof}

The following lemma will be used below to prove Lemma~\ref{l:plateau-heights}.
\begin{lemma} \label{l:plateau-heights-hamiltonian}
Suppose $a, b \in (0, \; 0.1]$ with $a \le 1.5b$.
Let $y_1(a, b)$ and $y_2(a, b)$ be the $y$-coordinates of the intersections of the stable manifolds of the left and the right saddle, respectively, with the transversal $x = - \frac \pi 2$.
Then,
\[
 \frac{\partial y_j}{\partial a} < -1, \qquad j = 1, 2.
\]
\end{lemma}
\begin{proof}
 The values of the Hamiltonian $H$ at $(-\pi/2, y_1)$ and $(-\pi/2, y_2)$ are the same as at the corresponding saddles. Denote these values by $H_1$ and $H_2$, respectively.
 Since $x=-\pi/2$ on the transversal, we have $H_i = b y_i + a \pi / 2$, and
 \[
 b \frac{\partial y_i}{\partial a} = \frac{\partial H_i}{\partial a} - \frac \pi 2.
 \]
 Using the formula for $H_i$ from Lemma~\ref{l:Hamilt-at-saddles}, we get
 \[
 \frac{\partial H_1}{\partial a} = \frac \pi 2 - \mu,
 \qquad
 \frac{\partial H_2}{\partial a} = -\frac \pi 2 + \mu,
 \qquad \text{where }
 \mu = \frac 1 2 \big( \arcsin(b+a) + \arcsin(b-a) \big).
 \]
 Hence,
 \[
 b\frac{\partial y_1}{\partial a} = - \mu,
 \qquad
 b\frac{\partial y_2}{\partial a} = -\pi  + \mu.
 \]
 As $|\mu| \le \frac \pi 2$, we we have $\frac{\partial y_2}{\partial a} \le \frac{\partial y_1}{\partial a} = -\mu/b$.
 As $(\arcsin x)' \ge 1$, we have
 \[
 \arcsin(b+a) + \arcsin(b-a) = \arcsin(b+a) - \arcsin(a-b) > 2b,
 \]
 and $\mu > b$.
 Thus, $\frac{\partial y_2}{\partial a} \le \frac{\partial y_1}{\partial a} < -1$.
\end{proof}

\section{Reduction to a torus flow - details} \label{s:Fenichel-details}
\begin{proof}[Sketch of a proof of Proposition~\ref{p:fenichel}]
This sketch follows the book~\cite{jones1995geometric}, and refers to theorems by Fenichel~\cite{fenichel} as stated in this book.
For clarity, we will first prove a weaker version where the required smallness of $\eps$ can depend on $a$ and $b$ and without checking that ${\bf f}$ smoothly depends on $a$ and $b$.

Rescaling the time by setting $t=\eps\tau$ we turn~\eqref{e:MR-x-y} into
\begin{equation}\label{e:flowtau}
{\bf x'} = \eps {\bf y}, \qquad {\bf y}' = - ( {\bf y} - {\bf v} ( {\bf x} ) ),  \ \ ^\prime = \frac{d}{d\tau} .
\end{equation}
Now for $\eps=0$ the manifold $\mathcal M_0$ defined by
${\bf y}={\bf v}({\bf x})$ consists of critical points. It is normally hyperbolic center manifold with $2$-dimensional contracting leaves.  According to Fenichel's Theorem 1, for sufficiently small $\eps$  there exists  an invariant manifold $\mathcal M_\eps$  that is smooth (including in $\eps$), and   $C^r$-close to $\mathcal M_0$ (where $r$ depends on the smallness of $\eps$).
The local stable manifold $W^s(\mathcal M_\eps)$ of $\mathcal M_\eps$ is $4$-dimensional, so it contains a neighborhood of $\mathcal M_\eps$. By Fenichel's Theorem 3, $W^s(\mathcal M_\eps)$ is foliated by stable manifolds of different points of $\mathcal M_\eps$, those are two-dimensional manifolds, and the exponential convergence holds.
Now, note that for $\eps=0$ the manifold $\mathcal M_0$ is globally attracting, this means that the global stable manifold $W^s(\mathcal M_\eps)$ is $\mathbb R^4$, and all $\mathbb R^4$ is foliated by two-dimensional stable fibers.
Finally, to get the dynamics on $\mathcal M_\eps$, let us write
$\mathcal M_\eps$ as a graph of a function
\begin{equation} \label{e:inv-manifold}
    {\bf y} = {\bf v}({\bf x}) + \eps {\bf f} ({\bf x}, \eps).
\end{equation}
The expression for $f$ will come out of the condition that this graph is invariant under the flow \ref{e:flowtau}. Thus differentiating \ref{e:inv-manifold} and using \ref{e:flowtau} we obtain
\[
    -(  \underbrace{{\bf y}-{\bf v}({\bf x})}_{\eps {\bf f} } )
    =
    \Big(
      \frac{\partial {\bf v}}{\partial {\bf x}}
      +
      O(\eps)
    \Big){\bf x}^\prime
    =
    \eps \frac{\partial {\bf v}}{\partial {\bf x}}{\bf v}({\bf x}) + O(\eps^2),
\]
yielding
\[
    {\bf f} = -\frac{\partial {\bf v}}{\partial {\bf x}} {\bf v} + O(\eps).
\]

To show that smallness of $\eps$ can be taken uniformly over all $a$ and $b$ with $|a|, |b| < M$, one can consider a six-dimensional system treating $a$ and $b$ as slow variables with $\dot a = 0$ and $\dot b = 0$, then obtain a four-dimensional slow mainfold for this system when $\eps$ is small enough, and then divide it into two-dimensional manifolds for different values of $a$, $b$.
As the four-dimensional slow mainfold is $C^r$, this also implies that ${\bf f}$ smoothly depends on $a$ and $b$.
\end{proof}

\textbf{Explicit formula for the perturbation.} In the appendices below we study the perturbed system~\eqref{e:perturbed}, which requires a more explicit formula for the perturbation ${\bf f}$.
The general formula
${\bf f} = - \frac{\partial {\bf v}}{\partial {\bf x}} {\bf v} + O(\eps)$
applied to the vector field~\eqref{e:fluid} yields
\begin{equation} \label{e:f}
 {\bf f} = \Big(
 \frac 1 2 \sin2x + a \cos x \cos y - b \sin x \sin y,
 \
 \frac 1 2 \sin2y + a \sin x \sin y - b \cos x \cos y
 \Big) + O(\eps).
\end{equation}
We will also often use a formula for the divergence of this vector field:
\begin{equation} \label{e:div}
\Div {\bf f} = \cos 2x + \cos 2y + O(\eps) = 2\cos(x+y)\cos(y-x) + O(\eps).
\end{equation}
Notably, the divergence does not depend on the parameters $a$ and $b$.

\section{Properties of the rotation number} \label{s:rotation_number}
In this appendix, we prove the statement of Remark~\ref{r:rho-rho} and Lemma~\ref{l:rot-numbers}.
\begin{proof}[Proof of Lemma~\ref{r:rho-rho}]
  Recall that we are assuming that the initial point $\tilde z_0$ has infinite future orbit.
  A standard fact on monotone degree one circle maps (possibly undefined at some points like the map $P$) is that there is $\rho \in \mathbb R$ such that the sequence $\tilde P^n(\tilde z_0) - \rho n$ is bounded.
  Indeed, set $a_n = \tilde P^n(\tilde z_0) - \tilde z_0$.
  Then $a_m + a_n - 1 < a_{n+m} < a_m + a_n - 1$ for any $m, n \in \mathbb N$: $a_{n+m} = a_n + \big(\tilde P^{m+n}(\tilde z_0) - \tilde P^n(\tilde z_0)\big)$, and the second term differs from $a_m$ by less than one due to monotonicity and degree one. This inequality implies that there is $\rho$ such that $\rho n - 1 < a_n < \rho n + 1$ by~\cite[Exercise~99]{polya2012problems}.

  The fact above implies that $\rho$ exists, and $\tilde P^n(\tilde z_0) = \rho n + O(1)$.
  Let us enumerate all intersections of the positive orbit ${\bf x}(t)$ of ${\bf x}_0$ with the lift of the transversal $\tilde T$ as ${\bf x_n}$, treating ${\bf x}_0$ as the zeroth intersection.
  Then $z({\bf x}_n) = \rho n + O(1)$ and $x({\bf x}_n) = \pi n + O(1)$.
  So,
  $y({\bf x}_n) = 2\pi z({\bf x}_n) + x({\bf x}_n) = (2\rho + 1)\pi n + O(1)$.
  Hence, the intersections of ${\bf x}(t)$ with $\tilde T$  are at a finite distance from the ray starting at ${\bf x}_0$ with the slope $2\rho+1$.
  Finally, let $\gamma$ be the piece of ${\bf x}(t)$ between ${\bf x}_0$ and ${\bf x}_1$, let $R$ be the curvilinear rectangle form by $\gamma$, $\gamma$ shifted up by $2\pi$, and vertical segments connecting the endpoints of these two curves. Then, each fragment of ${\bf x}(t)$ between two consecutive intersections with $\tilde T$ can be covered by a shift~\eqref{e:shifts} of $R$ and so has a bounded diameter. Hence, the whole ${\bf x}(t)$ lies within bounded distance from the ray above.
\end{proof}

\begin{proof}[Proof of Lemma~\ref{l:rot-numbers}]
The first two properties can be proved using the standard argument for the case of circle homeomorphisms (as written in~\cite[Appendix, Lemma~3]{palis2012geometric}) with minor modifications.
Let us now prove Property~3. Clearly, $\rho(s)$ is nondecreasing. The claim that $\rho(s)$ cannot be constant and irrational on any interval follows from Corollary~\ref{c:dim-H-key}.

The following auxiliary statement is needed to prove Property~4. 
\begin{lemma}
  $\rho(s) = \frac p q \in \mathbb Q$ if and only if for some flat spot $I_j$ we have $\tilde f^q(I_j) \in I_j + p$.
\end{lemma}
\begin{proof}
  The "if" part is clear: the point $f^q(I_j)$ is periodic, so that for this point the rotation number is $p/q$. To prove the ``only if'' part, consider the map $\tilde g(x) = \tilde f^q(x) - p$. As $\rho(\tilde f) = p/q$, we have $\rho(\tilde g) = 0$. Thus, $\tilde g$ has fixed points -- otherwise, we would have either $\tilde g(x) > x$ for all $x$ or $\tilde g(x) < x$, which would imply $\rho(\tilde g) > 0$ or $\rho(\tilde g) < 0$, respectively. In other words, the graph of $y=\tilde g(x)$ intersects the line $y=x$. The map $\tilde g$ is a map with flat spots expanding outside them, as follows from the chain rule. It is impossible that all intersections of its graph with $y=x$ are at the expanding parts (then the graph goes from below $y=x$ to above $y=x$ and cannot go back);  flat spots therefore intersect $y=x$. This gives a periodic orbit of $f$ with the period $q$, and we can take
  any flat spot of $f$ this orbit intersects as $I_j$.
\end{proof}

\noindent Now we are ready to prove Property~4. The set $\{\rho^{-1}(p/q)\}$ is closed, non-empty, and connected by monotonicity and continuity of $\rho$. So, it can be either a closed interval, as claimed in Property~4, or a point. Let us show that this set is not just one point. Suppose $\rho(s_0) = p/q$; then for some flat spot $I_j$ we have $\tilde f_{s_0}^q(I_j) \in I_j + p$. If $\tilde f_{s_0}^q(I_j)$ is not the right endpoint of $I_j$, we will have $\tilde f_{s_1}^q(I_j) \in I_j + p$ for some $s_1>s_0$ near  $s_0$. Otherwise, $\tilde f_{s_0}^q(I_j)$ is not the left endpoint, and this holds for $s_1$ close to $s_0$ and smaller than $s_0$.
\end{proof}

\section{Perturbed flow} \label{a:perturbed}
In this appendix, we establish several statements on the perturbed flow $\eqref{e:perturbed}$ used above: the first part of Proposition~\ref{p:circle-maps-family} and Lemmas~\ref{l:contract},~\ref{l:plateau-heights}, and~\ref{l:monotone}.
For brevity, we will refer to the two branches of the invariant manifold of a saddle as \emph{separatrices} throughout this appendix. Accordingly, in our notation, each saddle has two stable separatrices and two unstable ones.

\textbf{The following quantifiers apply to all statements in this appendix}, and will not be stated separately:
\begin{itemize}
    \item there exists $\gamma_0 > 0$ such that for any $a_0,~b_0 \in (0, \gamma_0)$ there exists $\eps_1 > 0$ depending on $a$ and $b$ such that the statement holds for any $a$, $b$, $\eps$ with $0 < \eps < \eps_1$ and $|a-a_0|, |b-b_0| < \eps_1$.
\end{itemize}
A compactness argument shows that $\eps_1$ can be taken uniformly if $a, b$ are in a fixed compact subinterval of $(0, \gamma_0)$, such as $[\delta, \gamma_0-\delta]$ for any $\delta$.
So, Proposition~\ref{p:circle-maps-family} and Lemmas~\ref{l:contract},~\ref{l:plateau-heights}, and~\ref{l:monotone} as stated in the main part of the paper follow from the same statements proved in the common assumptions of this Appendix.

Typically, it suffices to take $\gamma_0=0.1$, thus implying $0 < a, b < 0.1$ after taking small enough $\eps_1$.
We will implicitly assume $\gamma_0 = 0.1$ throughout the Appendix. However, sometimes we are only able to prove that $\gamma_0$ exists without an exlicit estimate, in which case we will write ``for small enough $a$ and $b$''.

\subsection{Repulsion from the cell centers} \label{s:repel}
Roughly speaking, trajectories of the perturbed system~\eqref{e:perturbed} starting inside a homoclinic loop of the unperturbed system~\eqref{e:fluid-ode} spiral away from a fixed point near the center of the loop towards the neighborhood of the heteroclinic loop of the unperturbed system.
In this subappendix, we state a lemma that formalizes this intuition and then derive several corollaries used later to describe the phase portrait of the perturbed system; we also  prove the first part of Proposition~\ref{p:circle-maps-family}.
\begin{lemma} \label{l:repel}
There exists a positive $\delta = \delta(a, b)$ such that the following holds.
Let $\gamma$ be a closed trajectory of the Hamiltonian system~\eqref{e:fluid-ode} or its homoclinic loop. Then,
\begin{equation}
  \iint_{\Int \gamma} \Div {\bf f} \;dxdy > \delta A(\gamma),
\end{equation}
where $A(\gamma)$ is the area bounded by $\gamma$.
\end{lemma}
  \begin{figure}[H]
      \centering
      \includegraphics[scale=2]{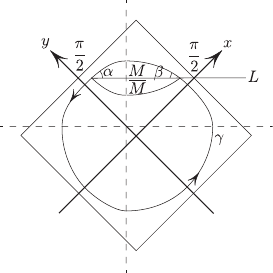}
      \caption{To the proof of Lemma~\ref{l:repel}. For convenience, the drawing is rotated by $  \pi/4 $.}
      \label{f:repel}
  \end{figure}
\begin{proof}
  Let us relax the implicit assumption $0 < a_0, b_0 < 0.1$ stated in the beginning of Appendix~\ref{a:perturbed} to $0 < |a_0|, |b_0| < 0.1$ to make the statement symmetric.

  By~\eqref{e:div} we have $\Div {\bf f} = \psi + O(\eps)$, where $\psi(x, y) = 2 \cos(x+y)\cos(y-x)$. Let us first prove that
  \begin{equation} \label{e:int-psi}
      \iint_{\Int \gamma} \psi \;dxdy > 0.
  \end{equation}
  As the vertical lines $x=\frac \pi 2 + \pi k$ and the horizontal lines $y=\frac \pi 2 + \pi k$ are transversal to the flow~\eqref{e:fluid-ode}, $\gamma$ is inside one of the squares
  $(-\frac \pi 2 + \pi k_1, \frac \pi 2 + \pi k_1) \times (-\frac \pi 2 + \pi k_2, \frac \pi 2 + \pi k_2)$.
  Without loss of generality, we will assume that $\gamma \subset (-\frac \pi 2, \frac \pi 2)^2$. Divide $\Int \gamma$ into four parts $\Int_1\gamma, \dots, \Int_4 \gamma$ lying inside each quadrant. We will prove that for each of the quadrants the integral over the corresponding part of
  $ \Int \gamma$ is positive, unless this part is empty. Note that in this lemma we do not require $a, b > 0$, assuming only that  $a, b \ne 0$ and $|a|, |b| < 0.1$. So, the statement is symmetric, and it is enough to prove this fact for $\Int_1 \gamma$ (and any allowed $a, b$):
  \begin{equation} \label{e:psi}
    \iint_{\Int_1 \gamma} \psi \;dxdy > 0, \qquad \text {provided that } \Int_1 \gamma \text{ is non-empty.}
  \end{equation}
  We have $\Int_1 \gamma \subset (0, \frac \pi 2)^2$. Denote by $L$ the line $x+y = \frac \pi 2$; this line is the diagonal of the square  $(0, \frac \pi 2)^2$, dividing it  in two parts (Figure~\ref{f:repel}), with $\psi < 0$ in the upper part  and with $\psi > 0$ in the lower part. Note also that $\psi$ is odd with respect to $L$ (indeed, for two symmetric w.r.t. $ L $  points the values of $ x-y $ are equal while the values of $  \pi/2 -(x+y) $ have opposite signs).
  If the whole of $\gamma$ is below $L$ then  $\psi$ is positive on $\Int_1 \gamma$, and~\eqref{e:psi} is clear.
  Along the Hamiltonian vector field~\eqref{e:fluid-ode}, we have (using product to sum formulas)
  \[
    \frac{d}{dt}(x+y) = \sin(x-y)+(a+b),
    \qquad
    \frac{d}{dt}(x-y) = -\sin(x+y)+(b-a).
  \]
  So, in Figure~\ref{f:repel} the vector field~\eqref{e:fluid-ode} is horizontal when $\sin(y-x)=a+b$ and vertical when $\sin(x+y)=b-a$,  drawn as two\footnote{Those equations also give other lines, but it is easy to see that a closed trajectory remianing in $(-\frac \pi 2, \frac \pi 2)^2$ cannot intersect them.} dashed lines in Figure~\ref{f:repel}.
  This means that in the first quadrant $\gamma$ first goes (using the orientation of Figure~\ref{f:repel}) up and left, then crosses the vertical dashed line, and then goes down and left, thus it intersects $L$ exactly twice (or is fully below $L$, this case is easy as $\psi > 0$ below $L$). Moreover, the angles $\alpha$ and $\beta$ formed by the part of $\gamma$ above $L$ and the segment of $L$ connecting two intersection points are acute. Denote by $M$ the part of $\Int_1 \gamma$ above $L$ and by $\bar M$ its reflection over $L$. As $\alpha$ and $\beta$ are acute, $\bar M \subset \Int_1 \gamma$. By symmetry, $\iint_{M \cup \bar M} \psi \;dxdy = 0$, and, as $\Int_1 \gamma \setminus (M \cup \bar M)$ is below $L$,
  \[
    \iint_{\Int_1 \gamma} \psi \;dxdy = \iint_{\Int_1 \gamma \setminus (M \cup \bar M)} \psi \;dxdy > 0.
  \]
  Finally, taking the sum over all quadrants yields~\eqref{e:int-psi}.

  Let us now use~\eqref{e:int-psi} to complete the proof of the lemma.
  First, let us prove that for small enough $\hat \delta$
  \begin{equation*}
    \iint_{\Int \gamma} \psi \;dxdy > \hat \delta A(\gamma).
  \end{equation*}
  Set $I(\gamma) = \iint_{\Int \gamma} \psi \;dxdy$.
  As the closed trajectories are nested, we can parametrize them by the area:
  $\gamma = \gamma(A)$, where $A \in [0, A_{\max}]$ with $A_{max}$ being the area bounded by the homoclinic loop.
  The function $I(A) = I(\gamma(A))$ is continuous. It is positive on $[0.001, A_{max}]$ and thus reaches some minimum $I_1$ there. Then we have $I(A) \ge \frac{I_1}{A_{\max}} A$ if $A \ge 0.001$. Finally, all bounded trajectories with the area less than $0.001$ lie inside the square $|x| + |y| < \frac \pi 2 - 0.01$ where $\psi$ is positive and bounded away from zero by some small constant $\kappa$.
  So, $I(A) > \kappa A$ on $[0, 0.001]$.

  It remains to note that
  \[
    \int_\gamma \Div {\bf f} \; dxdy = \int_\gamma \psi(x, y) + O(\eps) \; dxdy > 0.5 \hat \delta A(\gamma)
  \]
  provided that $\eps$ is small enough.
\end{proof}

\begin{corollary} \label{c:no-periodic-inside-loops}
  The perturbed system~\eqref{e:perturbed} has no periodic orbits or homoclinic loops inside or close to homoclinic loops of the unperturbed system.

  More formally, let $L_\delta$ denote the $\delta$-neighborhood of the region bounded by a homoclinic loop.
  Then, for any $c > 0$ for all small enough $\eps$   the perturbed system~\eqref{e:perturbed} has no periodic orbits or homoclinic loops contained inside $L_{c \eps}$ with smallness of $ \eps $ depending on $ c $.
\end{corollary}
\begin{proof}
  Assume the contrary: there exists a periodic orbit or a homoclinic loop $\gamma_\eps$ of the perturbed system contained inside $L_{c \eps}$. But then $\gamma_\eps$ is $O(\eps)$-close to a periodic orbit $\gamma$ of the Hamiltonian system.
  By the divergence theorem, we have $\iint_{\Int \gamma_\eps} \Div ({\bf v} + \eps {\bf f}) \; dxdy = 0$.
  Since near the elliptic fixed points of the unperturbed system the divergence of the perturbation is negative by~\eqref{e:div}, $\gamma_\eps$ cannot be very small so that there exists $S > 0$ independent of $\eps$ such that the area bounded by $\gamma$ is greater than $S$.
  But then
  \[
  0 = \eps^{-1} \iint_{\Int \gamma_\eps} \Div ({\bf v} + \eps {\bf f})
  =
  \iint_{\Int \gamma_\eps} \Div {\bf f}
  =
  \iint_{\Int \gamma} \Div {\bf f} + O(\eps)
  > \delta S + O(\eps) > 0
  \]
  when $\eps$ is small enough. This contradiction proves the corollary.
\end{proof}

\begin{corollary} \label{c:int-div-loop} Referring to Figure~\ref{f:homoclinic-open},
  the homoclinic loops of the unperturbed system are split ``outwards'':  trajectories starting inside the former loops near the loops escape them. The size of the  split, i.e., the distance between two separatrices a fixed distance away from the saddle is of order $\eps$. Moreover, one of the two stable separatrices of each saddle never intersects the transversal $T$ defined in Section~\ref{s:poinc} (as it is ``trapped inside the loop'' as in Figure~\ref{f:homoclinic-open}).
\end{corollary}
\begin{proof}
  Consider any line segment $L$ transversal to a homoclinic loop of the unperturbed Hamiltonian system. Let $S$ be the segment between the intersections of the stable and unstable separatrices of the perturbed system that are continuations of the former homoclinic loop of the unperturbed system (we consider intersections closest to the saddle along the separatrices).
  Consider the closed contour $B$ formed by $S$ together with the two pieces of separatrices of perturbed system connecting the saddle with the endpoints of $S$ (this is a standard construction called \emph{Bendixson bag}).
  Note that these separatrix pieces are $O(\eps)$-close to the homoclinic loops of the Hamiltonian system.
  Hence, the same argument as in the proof of Corollary~\ref{c:no-periodic-inside-loops} shows that integral of $\Div {\bf f}$ over the area bounded by $B$ is positive and of order $\eps$.
  It equals to the flux of ${\bf v} + \eps {\bf f}$ through $B$, so this flux is also positive and of order $\eps$.
  All this flux goes through $S$
  (as separatrices are integral curves of ${\bf v} + \eps {\bf f}$),
  and restricted to $S$ the vector field is close to a constant.
  Thus, this flux goes outside and the length of this segment (which measures separatrix splitting) is of order $\eps$.
  To prove the "moreover" part, we note that the stable manifold is trapped inside $B$ in backward time.
\end{proof}

\begin{corollary} \label{c:fixed-points}
  The perturbed system~\eqref{e:perturbed} on the torus has four fixed points: two saddles that are continuations of the saddles of the unperturbed flow ${\bf v}$, and two unstable foci that are continuations of the two centers of ${\bf v}$.
\end{corollary}
\begin{proof}
  A straightforward computation shows that when when $0 < a, b < 0.1$ the unperturbed Hamiltonian system has two saddles and two centers on the fundamental domain $[-\frac \pi 2, \frac \pi 2] \times [-\frac \pi 2, \frac {3\pi} 2]$; and the map ${\bf x} \mapsto {\bf v}({\bf x})$ has nonzero Jacobians at these points.
  As the perturbation $\eps {\bf f}$ is small in the $C^1$ topology, by the implicit function theorem these fixed points are smooth functions of $\eps$ near $\eps=0$, and there are no other fixed points in small neighborhoods of these points. This implies that there are no other fixed points at all for small $\eps$, as outside the neighborhoods above $|{\bf v}|$ is separated from zero.

  Let us now classify these four fixed points of the perturbed system. The saddles of the unperturbed system clearly remain saddles for the perturbed system as well.
  A straightforward computation shows that the centers of the unperturbed system are quite close to the points $(0, 0)$ and $(0, \pi)$: namely, one of them satisfies $|x| + |y| < \frac \pi 4$, and the other one satisfies $|x| + |y-\pi| < \frac \pi 4$. Hence, continuations of these centers for the perturbed system are in the areas where $\Div {\bf f} > 0$ by~\eqref{e:div}. For the centers, the two eigenvalues are purely imaginary; for the perturbed system these eigenvalues are close to the ones for the unperturbed system so they are complex and congugate to each other, and their sum is $\Div ({\bf v} + \eps{\bf f}) = \eps \Div {\bf f} > 0$.
  Hence, the real parts of both eigenvalues are positive, and centers of the unperturbed system are continued for the perturbed system as unstable foci.
\end{proof}

\begin{corollary} \label{c:transversal-or-separatrix}
  For the flow~\eqref{e:perturbed}, the positive orbit with the starting value of     $x \in [-\frac \pi 2, \frac \pi 2)$ intersects the transversal $\{x = \frac \pi 2\}$ unless the initial point is a fixed point or lies on the stable manifold of a saddle.
\end{corollary}
\begin{proof}
  Consider any trajectory $\gamma$ going from $\{x = -\frac \pi 2\}$ to $\{x = \frac \pi 2\}$. Let $U$ be the open set bounded by $\gamma$, $\gamma$ shifted up by $2\pi$, and the transversals $\{x = \pm\frac \pi 2\}$.
  We can shift the initial point vertically by $2 \pi k$ to be in this domain, denote this shifted point by ${\bf x}$.
  As the $x$-component of the vector field is positive on both vertical transversals, the positive orbit of ${\bf x}$ can only escape $U$ through $\{x = \frac \pi 2\}$.
  Suppose this does not happen, then the orbit stays in $U$. By the  Poincar\'e-Bendixson Theorem the $\omega$-limit set of ${\bf x}$ either is (1) a periodic orbit, (2) a fixed point, or (3) a union of fixed points, homoclinic, or heteroclinic orbits.

  We claim that there are no periodic orbits contained in $U$. Indeed, for a large enough $c$, any orbit that is outside homoclinic loops of the unperturbed system and stays $c\eps$-far from these loops reaches $\{x = \frac \pi 2\}$, as the orbit of the same initial point for the unperturbed system reaches this transversal, and the orbit for perturbed system stays $O(\eps)$-close to it. If the orbit is inside the loops or is $c\eps$-close to them, it cannot be periodic by Corollary~\ref{c:no-periodic-inside-loops}.

  Now, consider cases 2 and 3, when $\omega({\bf x})$ contains a fixed point. The fixed points of the perturbed system are described by Corollary~\ref{c:fixed-points} (note that $U$ is a fundamental domain).
  Assume ${\bf x}$ is not one of those fixed points. Then unstable foci cannot be in the $\omega$-limit set, and the only possibility left is that $\omega({\bf x})$ contains a saddle.
  Then either ${\bf x}$ is on a stable manifold of the saddle, which finishes the proof, or $\omega({\bf x})$ contains one of the unstable separatrices of the saddle as well, moving us to case (3).
  In case 3, only homoclinic or heteroclinic orbits connecting saddles are relevant, as
  unstable foci cannot belong to $\omega({\bf x})$. We claim that there are no such orbits contained in $U$. Homoclinic saddle loops are forbidden by Corollary~\ref{c:no-periodic-inside-loops}.
  Heteroclinic saddle connections between the two saddles in $U$ are forbidden by for the following reason: when $\eps=0$, there is a trajectory of ${\bf v}$ going from $x=-\frac \pi 2$ to $x = \frac \pi 2$ and separating the two saddles of ${\bf v}$ by Lemma~\ref{l:ham-sep}; this trajectory survives for the perturbed system when $\eps$ is small enough and hence it prohibits heteroclinic connection between the two saddles in $U$. Hence, the orbit of any point that is not a fixed point and does not lie on a stable separatrix leaves $U$ via its right side $x = \frac \pi 2$.
\end{proof}

\subsection{Phase portrait of the perturbed flow} \label{s:phase-portrait}
In this subappendix, we prove the first part of Proposition~\ref{p:circle-maps-family}.
We refer to the system with $a=a_0$, $b=b_0$ (recall the discussion in the beginning of Appendix~\ref{a:perturbed}), and $\eps=0$ as the {\it  unperturbed system} and the system with the parameters $a$, $b$, and $\eps$ satisfying $|a-a_0|, |b-b_0|$, $0 < \eps < \eps_1$ the {\it  perturbed system}.

\begin{proof}[Proof of the first part of Proposition~\ref{p:circle-maps-family}]

\noindent \emph{There are two saddles, two sources, and no other fixed points.}
A straightforward computation shows that when when $0 < a_0, b_0 < 0.1$, the unperturbed Hamiltonian system has two saddles and two centers on $[-\frac \pi 2, \frac \pi 2] \times [-\frac \pi 2, \frac {3\pi} 2]$; and the map ${\bf x} \mapsto {\bf v}({\bf x})$ has nonzero Jacobians at these points.
As the perturbation $\eps {\bf f} + (b-b_0, a-a_0)$ is $O(\eps_1)$-small in the $C^1$ topology, by the implicit function theorem these fixed points are smooth functions of $a$, $b$, and $\eps$ near $a=a_0$, $b=b_0$, $\eps=0$, and there are no other fixed points in small neighborhoods of these points. This implies that there are no other fixed points at all when the parameters are close enough to $(a_0, b_0, 0)$, as outside the neighborhoods above $|{\bf v}|$ is sepa
rated from zero.

Let us now classify these four fixed points of the perturbed system. The saddles of the unperturbed system clearly remain saddles for the perturbed system as well. A straightforward computation shows that one of  centers of the unperturbed system satisfies $|x| + |y| < \frac \pi 4$, while the other satisfies $|x| + |y-\pi| < \frac \pi 4$; and  $\Div {\bf f} > 0$ in both of these regions by~\eqref{e:div}. For the centers, the two eigenvalues are purely imaginary; for the perturbed system these eigenvalues are complex conjugate to each other and with positive real parts since
is $\Div ({\bf v} + \eps{\bf f}) > 0$ as $\Div {\bf v} = 0$.
Thus centers of the unperturbed system become   unstable foci of the perturbed system.

\medskip
\noindent \emph{The forward trajectory of any point that is not on a stable mainfold of a saddle and is not one of the sources intersects the transversal $T$.}  This is the conclusion of Corollary~\ref{c:transversal-or-separatrix} proved above.

\medskip
\noindent \emph{The reverse-time first return map is defined and continuous on the whole circle $T$ except for two closed intervals $I_1$, $I_2$. Each of these intervals is bounded by an intersection of the two unstable manifolds  of a saddle with $T$. One-sided limits of this map at the endpoints of such interval $I_j$ exist, and both coincide with the point $b_j$ of intersection of the stable manifold  of the same saddle with $T$.}

Recall $T$ is the vertical segment $x = -\frac \pi 2$, the lift of a circle on the torus.
By Corollary~\ref{c:int-div-loop}, one of the two stable separatrices of each saddle never intersects $T$.
Consider the forward-time first return map $P$ first. By Corollary~\ref{c:transversal-or-separatrix}, it is defined on the whole of $T$ except for the two points of intersection of the stable manifolds of the two saddles with $T$. Since these intersections are transversal for the unperturbed system, they persist for the perturbed system if  $\eps_1$ is small enough. Consider one of the saddles and the corresponding coordinate $y=b_i$ on the transversal where $P$ is undefined. Both unstable separatrices of this saddle intersect $T$ (they cannot coincide with a stable separatrix of another saddle as shown in the proof of Corollary~\ref{c:transversal-or-separatrix}). This means that left and right limits of $P$ at $b_i$ exist and are given by the intersections of the two unstable manifolds of the saddle with $T$.

Referring to Figure~\ref{f:homoclinic-open}, let    $b_1$ and $b_2$ be the $ y $-coordinates of the two points where $P$ is undefined and let $I_1$ be the closed   arcs  whose ends are the left and the right limits of $P$ at $b_1$, with $ I_2 $ defined similarly.
As $P$ is smooth on $S^1 \setminus \{b_1, b_2\}$, the inverse map $Q$ exists on $S^1 \setminus (I_1 \cup I_2)$. Moreover, one-sided limits at the endpoints of $I_i$ exist and are equal to $b_i$. Hence, setting $Q|_{I_i} = b_i$ turns $ Q $ into a continuous map with two flat spots.
\end{proof}

\subsection{Monotonicity of the flat spots heights} \label{a:plateau-heights}
In this subappendix, we prove Lemma~\ref{l:plateau-heights}. To this end, it is enough to prove the estimate on the derivative claimed in Lemma~\ref{l:plateau-heights} in the quantifiers stated in the beginning of Appendix~\ref{a:perturbed}. The lemma will follow by compactness as discussed in beginning of Appendix~\ref{a:perturbed}.

\begin{proof}[Proof of Lemma~\ref{l:plateau-heights}]
  Recall that "flat spots heights" $\tilde b_j(a, b, \eps)$ were defined as the values of $z = \frac{y-x}{2\pi}$ at the lifts to $\mathbb R$ of the images of the flat spots under the map $Q_s$,  Figure~\ref{f:homoclinic-open}); ``flat spots heights'' $\tilde b_j$ are the $y$-coordinates of the intersections of stable manifolds of the two saddles with the transversal.

  When $\eps = 0$, for any fixed $a_0, b_0 \in (0, \; 0.1)$ the functions $\tilde b_1$, $\tilde b_2$ are defined by Lemma~\ref{l:ham-sep}, and Lemma~\ref{l:plateau-heights-hamiltonian} claims that $\frac{\partial b_j}{\partial a} < -\frac 1 {2\pi}$ -- here we need to divide by $2\pi$ because Lemma~\ref{l:plateau-heights-hamiltonian} used $y$ as the coordinate on the transversal instead of $z$.
  As the separatrices smoothly depend on the parameters and intersect the transversal transversally, $\tilde b_1$ and $\tilde b_2$ are actually defined and $C^\infty$-smooth on some small neighborhood of $(a_0, b_0, 0)$ -- even allowing $\eps$ to be negative.
  As $\frac{\partial b_j}{\partial a}$ is continuous at $(a_0, b_0, 0)$, we have the weaker estimate $\frac{\partial b_j}{\partial a} < -0.15$ (note that $0.15 < \frac 1 {2\pi}$) for any $a$ and $b$ with $|a-a_0|, |b-b_0| < \eps_1$ and any positive $\eps < \eps_1$ when $\eps_1$ is small enough.
\end{proof}

\subsection{Contraction property of the first return map} \label{s:expand}

Before proving that   the first return map is contracting, we  state a general fact relating the derivative of a first return map with the divergence of the vector field.
Consider a $C^1$ vector field $v(x)$ on $\mathbb R^2$.
Let $x(t), \; t \in [t_1, t_2]$ be its trajectory.
Denote $x_1 = x(t_1)$, $x_2 = x(t_2)$.

We need the following well known (\cite[\S27.6]{arnold1992ordinary}) fact.
\begin{lemma}[Liouville's Theorem] \label{l:div}
Let $\psi$ be the time $t_2-t_1$ flow of $v$. Then
\begin{equation}
    \det D_{x_1} \psi = \exp\bigg(\int_{t_1}^{t_2} \Div v \; dt\bigg).
\end{equation}
\end{lemma}

\begin{corollary} \label{c:d-poincare}
Consider a transversal $T_1$ passing through $x_1$ and a transversal $T_2$ passing through $x_2$.
Let $s_1$ and $s_2$ be natural (arc length) parameters on $T_1$ and $T_2$ that are zero at $x_1$ and $x_2$.
Let $v_{n, 1}, v_{n, 2} > 0$ be the components of $v(x_1), v(x_2)$ normal to $T_1, T_2$.
Let $h: T_1 \mapsto T_2$ be the Poincar\'e map (we assume that the orientation of $T_1$ and $T_2$ is chosen so that $h$ is increasing).
Then
\begin{equation} \label{e:h-div}
    h'(0) = \frac{v_{n, 1}}{v_{n, 2}} \exp\bigg(\int_{t_1}^{t_2} \Div v \; dt\bigg).
\end{equation}
\end{corollary}
\begin{proof}[Informal proof]
Consider stationary flow of compressible fluid with velocity field $v$.
Let $\Delta I_1 \subset T_1$ be a small interval covering $x_1$, let $\Delta I_2 = h(\Delta I_1)$. Denote by $\mathcal T$ the corresponding trajectory tube between $\Delta I_1$ and $\Delta I_2$.

Take some small time $\Delta t$. During this time, the volume flowing into $\mathcal T$ will be $\Delta V_1 = \Delta t \cdot v_{n, 1} |\Delta I_1|$, and the volume flowing out of $\mathcal T$ will be $\Delta V_2 = \Delta t \cdot v_{n, 2} |\Delta I_2|$.
Small volume $\Delta V$ of liquid entering $\mathcal T$ exits it having the volume $J \Delta V$, where
\[
    J = \exp\bigg(\int_{t_1}^{t_2} \Div v \; dt\bigg),
\]
due to Lemma~\ref{l:div}.
Thus, the outgoing volume is $J$ times the incoming volume:
\[
    \Delta V_2 = J \Delta V_1, \ \ \hbox{or} \ \
    \Delta t \cdot v_{n, 2} |\Delta I_2| = J \Delta t \cdot v_{n, 1} |\Delta I_1|.
\]
This yields
\[
    h'(0) = \frac{|\Delta I_2|}{|\Delta I_1|} = J \frac{v_{n, 1}}{v_{n, 2}}.
\]

\end{proof}

\begin{proof}[Formal proof] Recall that $\psi$ denotes the time $t_2-t_1$ flow, $\psi(x_1) = x_2$, and $D_{x_1}\psi: T_{x_1} \mathbb R^2 \mapsto T_{x_1} \mathbb R^2$ denotes the differential of $\psi$ at $x_1$.
Let us pick the   basis of $T_{x_1}$ consisting of the vector $u_1 = v(x_1)/v_{n, 1}$ and of    $u_2=$ a unit vector tangent to $T_1$ chosen so that the basis is right-handed.
We pick the basis $ w_1, \ w_2 $ in $T_{x_2}$ in the same way.
Note that $ \det([u_1, u_2]) = 1$ and $\det([w_1, w_2]) = 1$.
This means that the determinant of $D_{x_1} \psi$ written in these bases will be
\[
    J = \exp\bigg(\int_{t_1}^{t_2} \Div v \; dt\bigg)
\]
by Lemma~\ref{l:div}.
As $D_{x_1} \psi$ maps $v(x_1)$ into $v(x_2)$, it maps $u_1 = v(x_1)/v_{n, 1}$ into $\frac{v_{n, 2}}{v_{n, 1}} w_1$. As $\det (D_{x_1} \psi) = J$, this means
\[
D_{x_1} \psi
=
\begin{pmatrix}
\frac{v_{n, 2}}{v_{n, 1}} & * \\
0 & \frac{J v_{n, 1}}{v_{n, 2}}
\end{pmatrix}.
\]
Now, let $\phi: T_1 \mapsto \mathbb R^2$ be the embedding and let $\pi: \mathbb R^2 \mapsto T_2$ be the projection onto $T_2$ along the trajectories of $v$. Using the bases in $T_{x_1}$ and $T_{x_2}$ chosen above, we have
\[
D_0 \phi: \mathbb R \mapsto T_{x_1}
=
\begin{pmatrix}
0 \\
1
\end{pmatrix},
\qquad
D_{x_2} \pi: T_{x_2} \mapsto \mathbb R
=
\begin{pmatrix}
0 & 1
\end{pmatrix}.
\]
We can write the Poincare map as $h = \pi \circ \psi \circ \phi$, and
\[
h' = D_{x_2} \pi \cdot D_{x_1} \psi \cdot D_0 \phi = \frac{J v_{n, 1}}{v_{n, 2}}.
\]
\end{proof}

Next, let us prove a lemma describing the integral of the divergence along trajectories realizing the \poincare map for our particular system~\eqref{e:perturbed}. As $\Div {\bf v} = 0$, this divergence is $\eps \Div {\bf f}$.
\begin{lemma} \label{e:int-div-poinc}
  For small enough $a$ and $b$ the following holds for all small enough $\eps$.
  Consider a solution of~\eqref{e:perturbed} starting at $t=t_0$  on the transversal $T$ defined in Section~\ref{s:poinc}, and let   $t_1$ be the time of first return to $T$. Then
  \begin{equation} \label{e:int-div-ineq}
    \int_{t_0}^{t_1} \Div {\bf f} \; dt < -C_d|\ln b| < 0,
  \end{equation}
  where the integral is taken along the  solution, and the positive constant $C_d$ does not depend on $\eps$, chosen initial data on $T$ and the parameters $a$ and $b$.
\end{lemma}
\begin{proof}
First, we need to establish the estimate
\[
    t_1 - t_0 > |\ln b|.
\]
Consider solution of~\eqref{e:perturbed} starting on $T$ (more precisely, on its lift to $\mathbb R^2$) with $x = -\frac \pi 2$.
Set $w(t) = x(t)+\frac \pi 2$, where $x(t)$ is the first coordinate of this solution.
By~\eqref{e:fluid-ode} we have
\[
\dot w = - \cos\Big(w - \frac \pi 2\Big)\sin(y) + b + O(\eps)
<
\Big|\cos\Big(w - \frac \pi 2\Big)\Big| + \frac b 2
<
w + \frac b 2.
\]
We have $w(t_0)=0$; the corresponding solution of $\dot w = w + \frac b 2$ is $w(t) = \frac b 2 (e^{t-t_0}-1)$. Solution of~\eqref{e:perturbed} returns to $T$ when $w=\pi$.
This gives
\[
  t_1-t_0 > \ln(2\pi/b + 1) > \ln (1 / b) = |\ln b|.
  \]

  Getting back to estimating the integral of the divergence, we will use formula~\eqref{e:div}:
  \begin{equation*}
  \Div {\bf f} = \cos (2x) + \cos (2y) + O(\eps) = 2\cos(x+y)\cos(y-x) + O(\eps).
\end{equation*}
The leading term $2\cos(x+y)\cos(y-x)$ is $\pi$-periodic, positive in the squares surrounding cell centers with the vertices at the midpoints of the sides of the cell (for the cell $[-\frac \pi 2, \frac \pi 2]^2$ this square is $\{|x|+|y| < \frac \pi 2\}$), and negative outside these squares. Let us call the squares where $2\cos(x+y)\cos(y-x)$ is positive \emph{central squares}.
We will use that $a$ and $b$ are small enough -- as assumed in the beginning of Appendix~\ref{a:perturbed}, we have $a, b < \gamma_0$ with $\gamma_0$ as small as needed. Then trajectories whose backward orbit is captured inside the former separatrix loops of the unperturbed system occupy most of the central squares, and only their corners of size $O(\gamma_0)$ can intersect trajectories of the perturbed system that realize \poincare map.
These corners are crossed during the time $O(\gamma_0)$; for the rest of the   time the divergence is negative, and at least half of that time the divergence is bounded away from zero.
Hence, the time average of the divergence is bounded from above by a negative constant; together with the estimate on the time $t_1-t_0$ this gives~\eqref{e:int-div-ineq}.

\end{proof}

Now, we are ready to prove the contraction of the first return map.
\begin{proof}[Proof of Lemma~\ref{l:contract}]
We use the formula~\eqref{e:h-div} for the derivative of $P$ applied to the vector field $v = {\bf v} + \eps {\bf f}$. As this formula uses natural parameter on the transversal, it can be applied to $P$ written using the coordinate $y$, which we denote $\hat P$. In~\eqref{e:h-div}, we have
$\frac{v_{n, 1}}{v_{n, 2}} < 1 + 3 \eps$: restricted to the transversal $T$ where $x=\frac \pi 2 + \pi k$, the normal component of ${\bf v}$ is $b$ and the normal component of $\eps {\bf f}$ is $-\eps b\sin(x)\sin(y)$ by~\eqref{e:f}, which is bounded by $\eps b$.
As $\Div {\bf v} = 0$, the divergence of $v$ is $\eps \Div {\bf f}$, and we get
\begin{equation*}
    \frac{d}{dz} P = \frac{d}{dy} \hat P < (1 + 3\eps) \exp\bigg(\eps \int_{t_1}^{t_2} \Div {\bf f} \; dt\bigg),
\end{equation*}
where the integral is along the trajectory of~\eqref{e:perturbed} connecting $y$ with $P(y)$. The integral is bounded from above by $-C_d |\ln b|$ according to Lemma~\ref{e:int-div-poinc}, so we get
\[
    \frac{d}{dz} P < (1 + 3 \eps)\exp(1-\eps C_d |\ln b|) < (1+3 \eps)(1 - 0.5 \eps C_d |\ln b|) < 1-\eps,
\]
as $b$ is small enough.
\end{proof}

\begin{remark} \label{r:sing-derivative}
   The derivative of the inverse-time \poincare map $Q$ has infinite singularities at the boundaries of each flat spot.
\end{remark}
\begin{proof}
  The jumps of the forward-time \poincare map $P$ happen at the points $b_1$ and $b_2$ (Figure~\ref{f:homoclinic-open}) where stable manifolds of the two saddles intersect the transversal.
  When the starting point on the transversal approaches $b_1$ or $b_2$, the time spent near the corresponding saddle goes to infinity, and the integral~\eqref{e:int-div-ineq} goes to $-\infty$ as the divergence is negative near both saddles.
  So, proof of Lemma~\ref{l:contract} implies that the derivative of $P$ has zero limit at the jump singularities. For its inverse $Q$, this means the derivative has infinite singularities at the boundaries of the flat spots.
\end{proof}

\subsection{Monotonicity with respect to the parameter} \label{a:monotone}
In this appendix, we prove Lemma~\ref{l:monotone}.
Throughout this subappendix, we will assume that $b$ is fixed; with $a$  variable, we need a notation that reflects the dependence of the vector field on $a$. To that end we
denote by ${\bf w}_a$ the vector field in the RHS of~\eqref{e:perturbed} with given $a$, writing  ${\bf w}_a = {\bf v}_a + \varepsilon {\bf f}_a$.
\begin{figure}[H]
    \centering
    \includegraphics[scale=1]{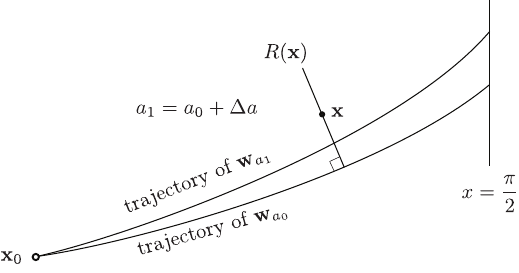}
    \caption{A tube.}
    \label{f:tube}
\end{figure}
\noindent Let us now introduce a technical notion of \emph{tube} (Figure~\ref{f:tube}) that we use to compare trajectories of ${\bf w}_a$ with different values of $a$.
To specify a tube, we fix some $a_0$, an initial point ${\bf x}_0 = (x_0, y_0)$ with $x_0 \in [-\pi/2, \pi/2)$, and $\Delta a$. We always assume that $\Delta a$ is small enough, depending on ${\bf x}_0$.
We also assume that ${\bf x}_0$ is such that its positive ${\bf w}_{a_0}$-orbit intersects the line $\{x=\pi/2\}$. Denote the trajectory of ${\bf x}_0$ by $l_0$. For a point ${\bf x}$ close to $l_0$, denote by $R({\bf x})$ the perpendicular to  $l_0$ passing through ${\bf x}$.
For small enough $\Delta a$ denote $a_1 = a_0 + \Delta a$, ${\bf w}_0 = {\bf w}_{a_0}$ and ${\bf w}_1 = {\bf w}_{a_1}$. Let $l_1$ be the trajectory of ${\bf w}_1$ starting at ${\bf x}_0$; this trajectory
 reaches $\{x=\pi/2\}$ as well if $\Delta a$ is sufficiently small. Let $\mathcal T$ be the \emph{tube} bounded by $l_0$, $l_1$, and the segment of the line $\{x=\pi/2\}$ between $l_0$ and $l_1$.
Note that $l_0$ and $l_1$ may intersect at some points other than ${\bf x}_0$ (the figure does not show such a case).

We assume that $\Delta a$ is so small that
\begin{equation} \label{e:90-deg}
  \text{for any } {\bf x} \in \mathcal T \text{, the angle between } {\bf w}_0({\bf x}) \text{ and } R({\bf x}) \text{ is in }[\pi/2 - 0.01, \pi/2 + 0.01].
\end{equation}

\begin{lemma} \label{l:bounded-length}
  The lengths of all trajectories of~\eqref{e:perturbed} realizing the Poincare map from $x=-\frac \pi 2$ to $\frac \pi 2$ are bounded (uniformly over all trajectories).
\end{lemma}
\begin{proof}
  Let $l(y)$ be the length of such a trajectory starting at $(-\pi/2, y)$. This function is continuous on the whole circle except the points $b_1$ and $b_2$ (Figure~\ref{f:homoclinic-open}), and at these two points it has one-sided limits equal to the lengths of the "limit" paths along separatrices. Hence, $l(y)$ is bounded.
\end{proof}

\noindent Let $R_{\mathcal T}(x) = R(x) \cap \mathcal T$.
\begin{lemma} \label{l:flux-bound}
There exists $c > 0$ such that, given any tube $\mathcal T$ and $x \in l_0$, the flux of ${\bf w}_0$ through $R_{\mathcal T}(x)= R(x) \cap \mathcal T$ is bounded by $c \Delta a$.
\end{lemma}
\begin{proof}
  Let $s$ be the natural parameter on $l_0$ with $s=0$ corresponding to ${\bf x}_0$, and let $\phi(s)$ be the flux of ${\bf w}_0$ through $R_{\mathcal T}(x)$, where $x \in l_0$ is given by $s$. Note that $\phi \ge 0$ by~\eqref{e:90-deg}.
  Let us show that
  \begin{equation} \label{e:d-phi}
    \frac{\partial}{\partial s} \phi \le c_1 (|\Delta a| + \phi)
  \end{equation}
  for some (uniform) constant $c_1$. This will imply the lemma because $\phi(0)=0$ and the lengths of trajectories $l_0$ starting with $x \in (-\pi/2, \pi/2)$ and crossing $\{x=\pi/2\}$ are uniformely bounded by Lemma~\ref{l:bounded-length}. Indeed, for $\tilde \phi = \phi/|\Delta a|$ equation~\ref{e:d-phi} implies
  $\frac{\partial}{\partial s} \tilde \phi \le c_1 (1 + \tilde \phi)$
  , so $\tilde \phi = O(1)$ and $\phi = O(\Delta a)$.

  To prove~\eqref{e:d-phi}, we take small $\Delta s$ and consider the part of $\mathcal T$ between $R_{\mathcal T}(s)$ and $R_{\mathcal T}(s+\Delta s)$.
  Let us apply the divergence theorem to the flux of ${\bf w}_0$ through the boundary of this domain.
  The flux through the $l_0$ side is zero. The flux of ${\bf w}_1$ through the $l_1$ side is zero, and so the flux of ${\bf w}_0 = {\bf w}_1 + O(\Delta a)$ is $O(\Delta a \Delta s)$.
  According to the divergence theorem we then have   $\phi(s+\Delta s) - \phi(s) + O(\Delta a \Delta s) = \iint \Div {\bf w}_0$.

  Consider now small neighborhoods of the saddles where $\Div {\bf w}_0$ is negative according to~\eqref{e:div}. In these neighborhoods the divergence theorem gives $\frac{\partial}{\partial s} \phi < O(\Delta a)$, which implies~\eqref{e:d-phi}. Outside these neighborhoods we claim that the width of $\mathcal T$ is $O(\phi)$. This is because $|{\bf w}_0|$ is separated from zero and its angle with the two sides of the rectangle formed by transversals is close to the right angle, so that
  the flux $\phi$ through $R_{\mathcal T}$ is of the same order as the width of $\mathcal T$. Thus, the area of the  part of $\mathcal T$ in question is $O(\phi \Delta s)$, and $\iint \Div {\bf w}_0 = O(\phi \Delta s)$, giving~\eqref{e:d-phi}.
\end{proof}

Given a tube, let ${\bf x}'_0$ and ${\bf x}'_1$ be the intersections of $l_0$ and $l_1$ with $\{ x = \pi/2 \}$. Let $\Delta H_0 = H_{a_0}({\bf x}'_0) - H_{a_0}({\bf x}_0)$ and $\Delta H_1 = H_{a_1}({\bf x}'_1) - H_{a_1}({\bf x}_0)$.

\begin{lemma} \label{l:d-da}
  \begin{equation*}
    y({\bf x}'_1) - y({\bf x}'_0) = \Delta a \frac{\pi/2 - x({\bf x}_0)} b + O(\Delta H_1 - \Delta H_0).
  \end{equation*}
\end{lemma}
\begin{proof}
  Set $H_0 = H_{a_0} ({\bf x}_0)$ and $H_1 = H_{a_1} ({\bf x}_0)$.
  We have $H_1 - H_0 = -x({\bf x}_0) \Delta a$.
  Set $H'_0 = H_{a_0} ({\bf x}'_0) = H_0 + \Delta H_0$ and $H'_1 = H_{a_1} ({\bf x}'_1) = H_1 + \Delta H_1$.
  We have $H'_1 - H'_0 = -\frac \pi 2 \Delta a  + b (y({\bf x}'_1) - y({\bf x}'_0))$. Thus,
  \[
   b \; (y({\bf x}'_1) - y({\bf x}'_0))
   = H'_1 - H'_0 + \frac \pi 2 \Delta a
   = -x({\bf x}_0) \Delta a + (\Delta H_1 - \Delta H_0) + \frac \pi 2 \Delta a.
  \]
\end{proof}

\begin{lemma} \label{l:Delta-H}
  $\Delta H_0$ is the flux of $\eps {\bf f}_{a_0}$ through $l_0$. Similarly, $\Delta H_1$ is the flux of $\eps {\bf f}_{a_1}$ through $l_1$.
\end{lemma}
\begin{proof}
  Set ${\bf f} = {\bf f}_{a_i}$, ${\bf w}={\bf w}_i$, $\Delta H = \Delta H_i$, and $l=l_i$, where $i=0, \;1$.
  The flux of $\eps {\bf f}$ through $l$ is
  $\int_l \big\langle \eps {\bf f}(t), -J({\bf w})dt \big\rangle$,
  where $t$ is the time, $\langle \cdot , \cdot \rangle$ is the inner product, and $J$ is the rotation by $-\pi/2$. Using  ${\bf w} = J \nabla H + \eps {\bf f}$ we have
  \[
  \int_l \langle \eps {\bf f}, -J(J \nabla H + \eps {\bf f})  \rangle dt
  =
  \underbrace{\int_l \langle \eps {\bf f}, -J^2 \nabla H  \rangle dt}_{\int_l \langle \nabla H, \eps {\bf f}  \rangle dt}
  -
  \underbrace{\int_l \langle \eps {\bf f}, J(\eps {\bf f})  \rangle dt}_{=0}
  =
  \int_l \dot H dt
  =
  \Delta H.
  \]

\end{proof}

\begin{remark} \label{r:change-H-one-wind}
  Together with Lemma~\ref{l:bounded-length}, Lemma~\ref{l:Delta-H} implies that the change of $H$ along a trajectory of~\eqref{e:perturbed} defining the Poincare map is $O(\eps)$.
\end{remark}

\begin{lemma} \label{l:dH-div}
  \[
  \Delta H_1 - \Delta H_0 = \iint_{\mathcal T} \eps \Div {\bf f}_{a_0} + O(\eps \Delta a).
  \]
  Here, the integral over the parts of $\mathcal T$ with $l_0$ above $l_1$ are included with positive sign, and over the parts with $l_0$ below $l_1$ have negative sign.
\end{lemma}
\begin{proof}
  This is the divergence theorem applied to the flux of $\eps {\bf f}_{a_0}$ through the boundary of $\mathcal T$.
  To be more precise, consider a contour formed by $l_1$ oriented forward, followed by the vertical segment connecting $\textbf{x}_1'$ with $\textbf{x}_0'$, and then completed by $l_0$ oriented backward.
  Applying the divergence theorem, we get that the integral of the divergence (with appropriate signs for different parts of $\mathcal T$) is equal to the sum of three fluxes: flux through $l_1$, flux through the vertical segment, and the flux through $l_0$ oriented backward. The last one is equal to $-\Delta H_0$ by Lemma~\ref{l:Delta-H}.
  Since ${\bf f}_{a_1} = {\bf f}_{a_0} + O(\Delta a)$ because ${\bf f} = {\bf f}(x, y, a, b, \eps)$ is $C^1$-smooth by Proposition~\ref{p:fenichel},
  the flux through $l_1$ is $\Delta H_1 + O(\eps \Delta a)$.
  By Lemma~\ref{l:flux-bound} together with the fact that near the line $x = \frac \pi 2$ the vector field ${\bf w}_0$ is close to $(b, a)$, the length of the vertical segment connecting ${\bf x}'_0$ with ${\bf x}'_1$ is $O(\Delta a)$. Thus, the flux of $\eps {\bf f}_{a_0}$ through this segment is $O(\eps \Delta a)$.
  This completes the proof of the lemma.
\end{proof}

\begin{lemma} \label{l:area-1}
  For any $C_1 > 0$ there exists $C_2 > 0$ such that the following holds.
  If the distance between $l_0$ and the saddles of ${\bf w}_0$ is greater than $C_1 \eps$, then
  \[
    |\Delta H_1 - \Delta H_0| < C_2 \sqrt{\eps} \Delta a.
  \]
\end{lemma}
\begin{proof}
  By Lemma~\ref{l:dH-div} it is enough to show that $\iint_{\mathcal T} \Div {\bf f}_{a_0} = O(\Delta a / \sqrt{\eps})$, or, that the area of $\mathcal T$ is bounded by $O(\Delta a / \sqrt{\eps})$.
  Let $U_1$ be the (possibly empty) part of $\mathcal T$ that is $\sqrt{\eps}$-close to saddles, and set $U_2 = \mathcal T \setminus U_1$.
  The length of the vector field ${\bf w}_0$ is bounded from below (up to a multiplicative constant) by $\sqrt{\eps}$ in $U_2$, and by $\eps$ in $U_1$.
  Thus, by Lemma~\ref{l:flux-bound} and~\eqref{e:90-deg} the width of $\mathcal T$ is $O(\Delta a / \eps)$ in $U_1$, and $O(\Delta a / \sqrt{\eps})$ in $U_2$.
  As the length of $U_1$ is $O(\sqrt{\eps})$, this means that the areas of $U_1$ and $U_2$ are both bounded by $O(\Delta a / \sqrt{\eps})$, as needed.
\end{proof}

Let $\hat P(y)$ be the \poincare map from the transversal $x = -\pi/2$ to the transversal $x = \pi/2$ written using the $y$   coordinate. Note that $\hat P(y)$ is not to be confused with $\tilde P$   whose argument is  $z = \frac{y-x}{2\pi}$ (in Section~\ref{s:poinc}) rather than $ y $.
\begin{corollary} \label{c:d-da-1}
  For any $C_1 > 0$ there exists $C_2 > 0$ such that the following holds.
  If $y$ is such that the trajectory connecting $(-\frac \pi 2, y)$ with $(\frac \pi 2, \hat P(y))$ is $C_1 \eps$-far from the saddles, then
  \[
    \frac{\partial}{\partial a} \hat P(y) = \frac{\pi}{b} + O(\sqrt{\eps}) > 0
  \]
  with $|O(\sqrt{\eps})| < C_2 \sqrt{\eps}$.
\end{corollary}
\begin{proof}
  Consider a tube starting at ${\bf x}_0 = (-\pi/2, y)$. We have
  \[
    \frac{\partial}{\partial a} \hat P(y) = \lim_{\Delta a \to 0} \frac{y({\bf x}_1') - y({\bf x}_0')}{\Delta a}.
\]
  By Lemma~\ref{l:d-da}   $\frac{y({\bf x}_1') - y({\bf x}_0')}{\Delta a} = \frac \pi b + O(|\Delta H_1 - \Delta H_0|/\Delta a)$. By Lemma~\ref{l:area-1} this gives the required estimate.
\end{proof}

The argument above implies monotonicity of the \poincare map far from its discontinuities which correspond to the trajectories bumping into the saddles. Let us continue with another argument showing monotonicity near the discontinuity points.
  \begin{figure}[H]
    \begin{center}
    \includegraphics[scale=0.3]{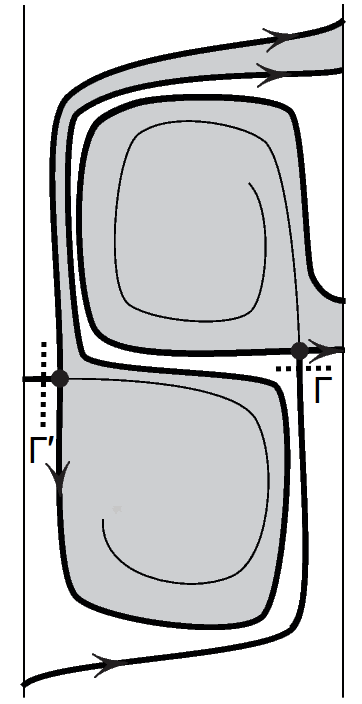}
    \caption{Transversals (dashed) used to study trajectories passing near saddles.}
    \label{f:monot_trans}
    \end{center}
  \end{figure}
\begin{lemma} \label{l:area-2}
  There exist $C_1, C_2 > 0$ such that the following holds.
  Consider a tube starting at a point ${\bf x}_0$ that is $C_1 \eps$-close to a saddle and lies on one of the transversals $\Gamma$ or $\Gamma'$ (Figure~\ref{f:monot_trans}). Then
  \[
    |\Delta H_1 - \Delta H_0| < C_2 \sqrt{\eps} \Delta a.
  \]
\end{lemma}
\begin{proof}
 Since the Hamiltonian takes on different values at the saddles,  a tube cannot be $C_1 \eps$-close to both saddles for small enough $\eps$.
  Also, there exists $C_1$ such that any tube starting $C_1 \eps$-close to a saddle cannot be $C_1 \eps$-close to the same saddle when it returns to this saddle. This is because this trajectory upon its return near the saddle cannot be in the gray-shaded zone (Figure~\ref{f:monot_trans}), and by Corollary~\ref{c:int-div-loop} the separatrix loop splits by $O(\eps)$ far from the saddle and thus by $O(\sqrt{\eps})$ near the saddle.

  Arguing as in the proof of Lemma~\ref{l:area-1}, we see that it is enough to show that $\iint_{\mathcal T} \eps \Div {\bf f}_{a_0}=O(\Delta a \sqrt{\eps})$. Denote by $\mathcal T_\eps$ the intersection of $\mathcal T$ with the $C_1\eps$-neighborhood of the saddle. One can show that
  $$\iint_{\mathcal T \setminus \mathcal T_\eps} \eps \Div {\bf f}_{a_0} = O(\Delta a \sqrt{\eps})$$
  as it was done in Lemma~\ref{l:area-1}.
  Now let us estimate the same integral over $\mathcal T_\eps$.
  Since $\Div \eps {\bf f}_{a_0} = \Div {\bf w}_0$, it is enough to estimate $\iint_{\mathcal T_\eps} \Div {\bf w}_0$.
  This can be done by applying the divergence theorem to the flux of ${\bf w}_0$ through the beginning of $\mathcal T$ bounded by some transversal $R(x)$ close to the boundary of $\mathcal T_\eps$ (recall that ${\bf x}_0$ is in $\mathcal T_\eps$).
  As the divergence is negative near the saddles, $\iint_{\mathcal T_\eps} \Div {\bf w}_0 < 0$. On the other hand, this integral is the sum of the fluxes through $l_0$, $l_1$, and $R$. The last of these is positive, so that the absolute value of the integral is bounded by the sum of the absolute values of the first two terms. The first of these is zero since  $l_0$ is a trajectory of ${\bf w}_0$. The second term is $O(\Delta a \eps)$ since the flux of ${\bf w}_1$ through $l_1$ is zero, while $||{\bf w}_1 - {\bf w}_0|| = O(\Delta a)$, and the length of $l_1$ is $O(\eps)$. Thus, the integral over $\mathcal T_\eps$ is $O(\Delta a \eps)$.
\end{proof}

\begin{corollary} \label{c:d-da-2}
  There exit $C_1, C_2 > 0$ such that the following holds.
  If $y$ is such that the trajectory connecting $(-\frac \pi 2, y)$ with $(\frac \pi 2, \hat P(y))$ passes $C_1 \eps$-near one of the saddles, then
  \[
    \frac{\partial}{\partial a} \hat P > \frac{L}{b} + O(\sqrt{\eps}) > 0
  \]
  with $|O(\sqrt{\eps})| < C_2 \sqrt{\eps}$.
  Here, $L = x_L + \frac \pi 2 = \frac \pi 2 - x_R > 0$, where $x_L$ and $x_R$ are the horizontal coordinates of the left and the rigth saddles, respectively.
\end{corollary}
  \begin{proof}
  Consider a vertical transversal $\Gamma'$ and a horizontal transversal $\Gamma$ shown in Figure~\ref{f:monot_trans} that are at the distance $C_\Gamma \eps$ from the saddles, where the constant $C_\Gamma$ is chosen to be so large that $\frac {\partial H}{\partial x} < 0$ restricted to $\Gamma$.\footnote{
  This can be done as $-\frac{\partial H}{\partial x}$ is the vertical component of the Hamiltonian vector field ${\bf v} = ({\bf v} + \eps {\bf f}) + O(\eps)$, and the vertical component of ${\bf v} + \eps {\bf f}$ is positive and larger than $0.01 C_\Gamma \eps$ because $\Gamma$ is below the saddle.}
  Let $\gamma$ denote a trajectory that realizes \poincare map and passes $O(\eps)$-close to one of the saddles; such trajectory crosses $\Gamma'$ or $\Gamma$.
  Consider the tube $\mathcal T$ starting at the point of intersection of $\gamma$ with $\{x=-\pi/2\}$.
  Consider also an auxiliary tube $\tilde {\mathcal T}$ starting at the point $\tilde {\bf x}_0$ where $\gamma$ intersects $\Gamma'$ or $\Gamma$.
  Denote the points where the two trajectories bounding $\tilde {\mathcal T}$ cross $\{x=\frac \pi 2\}$ by $\tilde {\bf x}'_0$ and $\tilde {\bf x}'_1$.
  By construction, $\tilde {\bf x}'_0$ = ${\bf x}'_0$.

  First, consider the case when $\gamma$ passes near the left saddle,
  in which case it crosses $\Gamma'$.
  For definiteness, suppose $\Delta a > 0$.
  Since before the crossing $\gamma$ goes from left to right and ${\bf w}_{a_1} - {\bf w}_{a_0}$ is close to a
  constant vector field $(0, \Delta a)$ pointing up, before the transversal crossing $l_1$ is above $l_0$.
  This means that
  $y({\bf x}'_1) > y(\tilde {\bf x}'_1)$,
  and
  $y({\bf x}'_1) - y({\bf x}'_0) > y(\tilde {\bf x}'_1) - y(\tilde {\bf x}'_0)$.
  Thus we have (taking into account $\Delta a > 0$)
  $\frac{\partial}{\partial a} \hat P(y) \ge \lim_{\Delta a \to 0} \frac{y(\tilde {\bf x}_1') - y(\tilde {\bf x}_0')}{\Delta a}$.
  By Lemma~\ref{l:d-da} we can continue
  \[
    \frac{\partial}{\partial a} \hat P(y) \ge \lim_{\Delta a \to 0} \frac{y(\tilde {\bf x}_1') - y(\tilde {\bf x}_0')}{\Delta a}
    \ge
    - \frac {\pi/2 - (x_s  + O(\eps))} {b} + O(|\Delta \tilde H_1 - \Delta \tilde H_0|/\Delta a),
  \]
  where $x_s$ is the $x$-coordinate of the left saddle, $x_s \approx -\frac \pi 2$.
  By Lemma~\ref{l:area-2} this gives the required estimate.

  Now, suppose that $\gamma$ passes near the right saddle,
  thus crossing
  $\Gamma$.
  For defineteness, suppose $\Delta a > 0$.
  First, we prove that the intersection of $l_1$ with $\Gamma$ is to the left of the intersection of $l_0$ with $\Gamma$.
  Recall that $\tilde {\bf x}_0$ and $\tilde {\bf x}_1$ denote the intersections of $l_0$ and $l_1$, respectively, with $\Gamma$.
  Set $\tilde \Delta H_0 = H_0(\tilde {\bf x}_0) - H_0({\bf x}_0)$ and
  $\tilde \Delta H_1 = H_1(\tilde {\bf x}_1) - H_1({\bf x}_0)$.
  Arguing as in Lemma~\ref{l:area-1}, we get the estimate
  $\tilde \Delta H_1 - \tilde \Delta H_0 = O(\sqrt{\eps} \Delta a)$.
  To compare the two trajectories, we estimate $H_1(\tilde {\bf x}_1) - H_1(\tilde {\bf x}_0)$. We have
  \begin{align*}
  \begin{split}
    H_1(\tilde {\bf x}_0)
    &=
    (H_1 - H_0)(\tilde {\bf x}_0) + H_0(\tilde {\bf x}_0)
    =
    (H_1 - H_0)(\tilde {\bf x}_0) + H_0({\bf x}_0) + \tilde \Delta H_0 =\\
    &=
    (H_1 - H_0)(\tilde {\bf x}_0) - (H_1 - H_0)({\bf x}_0) + H_1({\bf x}_0) + \tilde \Delta H_0 =\\
    &=
    -\Delta a \big(x(\tilde {\bf x}_0) - x({\bf x}_0)\big) + (\tilde \Delta H_0 - \tilde \Delta H_1)  + H_1(\tilde {\bf x}_1) .
  \end{split}
\end{align*}
  This gives
  $H_1(\tilde {\bf x}_1) = H_1(\tilde {\bf x}_0) + \Delta a \big(x(\tilde {\bf x}_0) - x({\bf x}_0)\big) + O(\sqrt{\eps}\Delta a) > H_1(\tilde {\bf x}_0)$.
  Recall that by our choice of the transversals   we have $\frac{\partial H_1}{\partial x} < 0$ on $\Gamma$.
  This implies $x(\tilde {\bf x}_1) < x(\tilde {\bf x}_0)$.

  Now, we can argue as in the first case considering auxiliary tube $\tilde {\mathcal T}$ starting at $\tilde {\bf x}_0$. This gives (using $x_s$ for the $x$-coordinate of the right saddle)
  \begin{align*}
  \begin{split}
    \frac{\partial}{\partial a} \hat P(y) &\ge \lim_{\Delta a \to 0} \frac{y(\tilde {\bf x}_1') - y(\tilde {\bf x}_0')}{\Delta a}
    \ge
    - \frac {\pi/2 - (x_s  + O(\eps))} {b} + O(|\Delta \tilde H_1 - \Delta \tilde H_0|/\Delta a)\ge\\
    &\ge
     \frac L b + O(\sqrt{\eps}).
  \end{split}
  \end{align*}
\end{proof}

\begin{proof}[Proof of Lemma~\ref{l:monotone}]
  Together, Corollaries~\ref{c:d-da-1} and~\ref{c:d-da-2} imply that $\frac{\partial}{\partial a}\hat P(y) > 0$ for any $y$ such that the \poincare map is defined. In Lemma~\ref{l:monotone} the \poincare map $\tilde P$ was written using the coordinate $z=\frac{y-x}{2\pi}$, and we have $\frac{\partial}{\partial a}\tilde P(y) = \frac{1}{2 \pi} \frac{\partial}{\partial a}\hat P(y) > 0$.
\end{proof}

\end{appendices}

\newpage

\printbibliography

\bigskip

\bigskip

\noindent Mark Levi,

\noindent {\small Department of Mathematics,}

\noindent {\small Pennsylvania State University,}

\noindent{\small University Park, State College, PA 16802}

\noindent {\footnotesize{E-mail : mxl48@psu.edu}}

\vskip 5mm

\noindent Alexey Okunev,

\noindent {\small Department of Mathematics,}

\noindent {\small Pennsylvania State University,}

\noindent{\small University Park, State College, PA 16802}

\noindent {\footnotesize{E-mail : abo5297@psu.edu}}

\end{document}